\documentclass[leqno,12pt]{amsart}
\usepackage{amsmath,amstext,amssymb,amsopn,amsthm,mathrsfs}

\numberwithin{equation}{section}

\allowdisplaybreaks

\DeclareMathOperator{\domain}{Dom}

\newtheorem{theor}{Theorem}[section]
\newtheorem{propo}[theor]{Proposition}
\newtheorem{lemma}[theor]{Lemma}
\newtheorem{corollary}[theor]{Corollary}

\newtheorem*{rem*}{Remark}

\newcommand{\N}{\mathbb{N}^d}
\newcommand{\R}{\mathbb{R}^d}
\newcommand{\C}{\mathbb{C}}

\begin{document}

\title[Negative powers of Laguerre operators]
	{Negative powers of Laguerre operators}


\author[A. Nowak]{Adam Nowak}
\address{Adam Nowak, \newline
			Instytut Matematyczny,
      Polska Akademia Nauk, \newline
      \'Sniadeckich 8,
      00--956 Warszawa, Poland \newline
			\indent and \newline
			Instytut Matematyki i Informatyki,
      Politechnika Wroc\l{}awska,       \newline
      Wyb{.} Wyspia\'nskiego 27,
      50--370 Wroc\l{}aw, Poland      
      }
\email{Adam.Nowak@pwr.wroc.pl}

\author[K. Stempak]{Krzysztof Stempak}
\address{Krzysztof Stempak,     \newline
      Instytut Matematyki i Informatyki,
      Politechnika Wroc\l{}awska,       \newline
      Wyb{.} Wyspia\'nskiego 27,
      50--370 Wroc\l{}aw, Poland      }  
\email{Krzysztof.Stempak@pwr.wroc.pl}

\begin{abstract}
We study negative powers of Laguerre differential operators in $\R$, $d\ge1$. 
For these operators we prove two-weight $L^p-L^q$ estimates, with ranges of $q$ depending 
on $p$. The case of the harmonic oscillator (Hermite operator) has recently 
been treated by Bongioanni and Torrea by using a straightforward
approach of kernel estimates. Here these results are applied in certain Laguerre settings. 
The procedure is fairly direct for Laguerre function expansions of Hermite type, 
due to some monotonicity properties of the kernels involved. 
The case of Laguerre function expansions of convolution type is less straightforward. 
For half-integer type indices $\alpha$ we transfer the desired results from the Hermite setting 
and then apply an interpolation argument based on a device we call the {\sl convexity principle} 
to cover the continuous range of $\alpha\in[-1/2,\infty)^d$. Finally, we investigate negative powers 
of the Dunkl harmonic oscillator in the context of a finite reflection group acting on $\R$ and 
isomorphic to $\mathbb Z^d_2$. The two weight $L^p-L^q$ estimates we obtain in this setting are essentially 
consequences of those for Laguerre function expansions of convolution type.
\end{abstract}

\maketitle

\section{Introduction} \label{sec:intro}

Consider the fractional integral operator (also referred to as the Riesz potential) 
\begin{equation*}
I^\sigma f(x)=\int_{\R}\frac1{\|x-y\|^{d-\sigma}}f(y)\,dy, \qquad x\in \R,
\end{equation*}
$0<\sigma<d$, defined for any function $f$ for which the above integral is convergent $x$-a.e.;
for instance, $f\in L^p(\R)$ with $1\le p<d\slash \sigma$ is good enough. 

Then, with an appropriate constant $c_\sigma$,
\begin{equation*}
(-\Delta)^{-\sigma \slash 2}f=c_\sigma I^\sigma f, \qquad f\in \mathcal S(\R),
\end{equation*}
where $\Delta=\sum_{j=1}^d \partial_j^2$ is the standard Laplacian in
$\R$, $d\ge1$, and  the negative power $(-\Delta)^{-\sigma \slash 2}$ 
is defined in $L^2(\R)$ by means of the Fourier transform. 

A classical result concerning $I^{\sigma}$ is the following, see e.g. \cite{Duo,St}.
\begin{theor}[Hardy-Littlewood-Sobolev] \label{thm:HLS}
Let $0<\sigma<d$, $1\le p<\frac d\sigma$ and $\frac1q=\frac1p-\frac{\sigma}d$. 
Then for $p>1$ we have the strong type $(p,q)$ estimate
$$
\|I^\sigma f\|_q \lesssim \|f\|_p, \qquad f\in L^p(\R),
$$
while for $p=1$ the weak type $(1,q)$ estimate  holds,
$$
|\{x\in \R\colon |I^\sigma f(x)|>\lambda\}|\lesssim \bigg(\frac{\|f\|_1}{\lambda}\bigg)^q, 
\qquad \lambda>0,\quad f\in L^1(\R).
$$
\end{theor}

The Hardy-Littlewood-Sobolev theorem was extended to a two-weight setting in \cite{SW}. 
\begin{theor}[E. M. Stein and G. Weiss] \label{thm:StWe}
Let $0<\sigma<d$, $1<p\le q<\infty$, $a<d/p'$, $b<d/q$, $a+b\ge0$ and 
$\frac1q=\frac1p-\frac{\sigma-a-b}d$. Then 
$$
\big\|I^\sigma f\big\|_{L^q(\|x\|^{-bq})}\lesssim \big\|f\big\|_{L^p(\|x\|^{ap})}, 
\qquad f\in L^p(\R, \|x\|^{ap}).
$$
\end{theor}
Note that the conditions $\frac1q=\frac1p-\frac{\sigma}d$ or 
$\frac1q=\frac1p-\frac{\sigma-a-b}d$ appearing in the above theorems are in fact necessary and
forced by a homogeneity type argument.

Numerous analogues of the Euclidean fractional integral operator were investigated in various settings,
including spaces of homogeneous type, orthogonal expansions, etc. For instance, in the seminal article of
Muckenhoupt and E. M. Stein \cite{MS} the case of ultraspherical expansions was treated. Gasper and Trebels
(and one of the authors of the present article) studied fractional integration for one dimensional Hermite
and Laguerre function expansions \cite{GST,GT}; 
the Laguerre case was also considered by Kanjin
and E. Sato \cite{KS}. Recently, Bongioanni and Torrea \cite{BT1} obtained $L^p - L^q$ estimates for
negative powers of the harmonic oscillator. In a more general context Bongioanni, Harboure and Salinas
\cite{BHS} investigated weighted
inequalities
for negative powers of Schr\"odinger operators with weights satisfying the reverse H\"older inequality.
Our present results generalize significantly those of \cite{GST,GT,KS}.

In this paper we focus on negative powers of ``Laplacians'' associated to multi-dimensional
Laguerre function expansions. 
For these operators we prove  two-weight $L^p-L^q$ estimates in the spirit of Theorem \ref{thm:StWe}.
Such estimates are of interest, for instance in the study of 
higher order Riesz transforms or Sobolev spaces related to Laguerre expansions.
In all the cases we discuss, spectra of self-adjoint extensions of the considered operators are discrete,
separated from zero, subsets of $(0,\infty)$. Hence negative powers of them are well defined in appropriate
$L^2$ spaces just by means of the spectral theorem.
The relevant extensions to weighted $L^p$ spaces of the negative powers are given by suitable
integral representations. The emerging integral operators are called the \emph{potential operators}
(sometimes also referred to as the \emph{fractional integral operators}).
Also, we take an opportunity to slightly enhance the result obtained by Bongioanni and Torrea for the
harmonic oscillator, by stating and proving a weighted counterpart (with power weights) to their result.

In the Laguerre case we consider two different systems of Laguerre functions, 
$\{\varphi_k^{\alpha}\}$ and $\{\ell_k^{\alpha}\}$.
The first one leads to so-called Laguerre function expansions of Hermite type.
It occurs that to some extent in this setting the problem of $L^p-L^q$ estimates for the potential operator
almost reduces to the Hermite case. This is due to the fact that the heat
kernels corresponding to different multi-indices of type $\alpha\in[-1/2,\infty)^d$
possess certain monotonicity
property with respect to $\alpha$. Thus it suffices to consider only the specific multi-index
$\alpha_o=(-1/2,\ldots,-1/2)$, which corresponds to Hermite function expansions. 

The second system of Laguerre functions
is related to so-called Laguerre expansions of convolution type. 
In this case our approach is quite different and in fact a more involved analysis is necessary.
We first deal with half-integer multi-indices $\alpha$ and transfer the desired results from
the Hermite setting. Then, to cover the continuous range of $\alpha\in[-1/2,\infty)^d$, 
we derive certain interpolation
argument, which we call the \emph{convexity principle}.
This method is of independent interest and can be applied in other situations.

The organization of the paper is the following.
In Section \ref{sec:herm} we gather known facts concerning
the potential kernel and the potential operator related to
the harmonic oscillator. Then we state and prove a two-weight 
$L^p-L^q$ estimate for the Hermite potential operator  in the spirit of Theorem \ref{thm:StWe}.
In Section \ref{sec:lag_herm} we discuss potential operators associated to Laguerre function expansions
of Hermite type. Section \ref{sec:lag} is devoted to Laguerre function expansions of convolution type.
Section \ref{sec:aver} establishes the convexity principle which allows to give proofs of the main
results of Section \ref{sec:lag}. In Section \ref{sec:dunkl} we take an opportunity to study
negative powers of the Dunkl harmonic oscillator in the context of a finite reflection group acting on
$\R$ and isomorphic to $\mathbb{Z}_2^d$. The results of this section contain as special cases 
those of Sections \ref{sec:herm} and \ref{sec:lag}, 
and are strongly connected with the estimates of Section \ref{sec:lag}.
Finally, in Section \ref{sec:final} we gather various additional observations and remarks.
Comments explaining how our present results generalize those of \cite{GST,GT,KS} are located
throughout the paper.

We use a standard notation with essentially all symbols referring to
either $\R$ or $\R_{+}=(0,\infty)^d$, $d\ge1$, depending on the context. 
Thus $\Delta$ denotes either the Laplacian in $\R$, or its restriction to $\R_{+}$, and $\|\cdot\|$
stands for the Euclidean norm. By $\langle f,g\rangle$ we denote
$\int_{\R}f(x)\overline{g(x)}\,dx$ (or the same, but with the integration restricted to $\R_{+}$) 
whenever the integral makes sense. For a nonnegative weight function $w$ on either $\R$ or $\R_+$, 
by $L^p(\R,w)$ or $L^p(\R_+,w)$, $1\le p\le \infty$, or simply by $L^p(w)$, we denote the 
usual Lebesgue spaces related to the measure $dw(x)=w(x)dx$ 
(we will often abuse slightly the notation and use the same symbol $w$ to 
denote the measure induced by a density $w$). If $w\equiv1$ we simply write $L^p(\R)$ or $L^p(\R_+)$.
Beginning from Section \ref{sec:lag}, Lebesgue measure $dx$ on  $\R_{+}$ is replaced by $\mu_\alpha(dx)$,
where $\alpha\in(-1,\infty)^d$ is a multi-index, hence some symbols previously related to $dx$
are then related to $\mu_\alpha(dx)$.
Similar situation occurs in Section \ref{sec:dunkl} where $dx$ on $\R$ is replaced by $w_{\alpha}(dx)$.

If $k\in\N$, $\mathbb N=\{0,1,\ldots\}$, then $|k|=k_1+\ldots+k_d$ is the length of $k$. 
The notation $X\lesssim Y$ will be used to indicate that $X\leq CY$ with a positive constant $C$
independent of significant quantities. We shall write $X \simeq Y$ when simultaneously 
$X \lesssim Y$ and $Y \lesssim X$.
Given $1 \le p \le \infty$, $p'$ denotes its adjoint, $1\slash p + 1\slash p' =1$.

\section{Negative powers of the harmonic oscillator} \label{sec:herm}
The multi-dimensional Hermite functions $h_k(x)$, $k\in \N$, are given by tensor products
$$
  h_k(x)=\prod_{i=1}^dh_{k_i}(x_i),\qquad x= (x_1, \ldots ,x_d)\in \R,
$$
where $h_{k_i}(x_i)=(\pi^{1/2}2^{k_i}{k_i}!)^{-1/2}H_{k_i}(x_i) e^{-x_i^2/2}$,
and $H_n$ denote the Hermite polynomials of degree $n\in \mathbb N$, cf{.} \cite[p.\,60]{Leb}. The system
$\{h_k : k \in \N\}$ is a complete orthonormal system in $L^2(\R)$. It consists of eigenfunctions of
the $d$-dimensional harmonic oscillator 
$$
\mathcal{H}=-\Delta+\|x\|^2,
$$
$\mathcal{H} h_k = \lambda_kh_k$, $\lambda_k=2|k|+d$.  
We shall denote by the same symbol the natural self-adjoint 
extension of $\mathcal{H}$, whose spectral resolution is given by the $h_k$ and $\lambda_k$,
see \cite{ST1}. The integral kernel of the Hermite semigroup
$\{e^{-t\mathcal{H}}: t>0\}$ is known explicitly (see \cite{ST2} for this symmetric form of the kernel),
\begin{align*}
G_t(x,y)&=\sum_{n=0}^\infty e^{-(2n+d)t}\sum_{|k|=n}h_k(x)h_k(y)\\
&=\big(2\pi\sinh(2t)\big)^{-d/2}\exp\bigg(-\frac14\Big[\tanh( t)\|x+y\|^2+\coth(t)\|x-y\|^2 \Big]\bigg).
\end{align*}

Given $\sigma>0$, consider the negative power $\mathcal{H}^{-\sigma}$. 
In view of the spectral theorem, it is expressed on $L^2(\R)$ by the spectral series,
\begin{equation}\label{neg}
\mathcal{H}^{-\sigma}f= \sum_{k \in \mathbb{N}^d} (2|k|+d)^{-\sigma}	\langle f,h_k\rangle \,h_k.
\end{equation}
Notice that $\mathcal{H}^{-\sigma}$ is a contraction on $L^2(\R)$ for any $\sigma>0$.

Motivated by the formal identity
$$
\mathcal{H}^{-\sigma}=\frac1{\Gamma(\sigma)}\int_0^\infty e^{-t\mathcal{H}} t^{\sigma-1}\,dt,
$$
it is natural to introduce the potential kernel 
\begin{equation}\label{ker}
\mathcal{K}^{\sigma}(x,y)= \frac1{\Gamma(\sigma)}\int_0^\infty G_t(x,y) t^{\sigma-1}\,dt.
\end{equation}
It follows from the decay of $G_t(x,y)$, as $t\to\infty$ and $t\to0^+$, that for $\sigma >d/2$ the integral
in \eqref{ker} is convergent for every $x,y\in\R$, while for $0<\sigma\le d/2$ the integral converges
provided that $x\neq y$.

Define the auxiliary convolution kernel $K^{\sigma}(x)$, $ x\in\R\backslash \{0\}$, by
$$
K^{\sigma}(  x)=\exp(-\|x\|^2/8), \qquad \|x\|\ge1,
$$
and, for  $\|x\|<1$,
\begin{equation*}
K^{\sigma}(x)=
\begin{cases}
1, & \quad \sigma>d/2, \\
\log( 4\slash \|x\|), & \quad \sigma=d/2, \\
\|x\|^{2\sigma-d}, & \quad \sigma<d/2.
\end{cases}
\end{equation*}
It is immediately seen that $K^{\sigma}\in L^1(\R)$ for all $\sigma>0$. 
Moreover, if $\sigma > d/2$, then $K^{\sigma}\in L^r(\R)$ for each $1\le r\le \infty$, 
if $\sigma=d\slash 2$, then $K^{\sigma}\in L^r(\R)$ for $1\le r< \infty$,
while for $\sigma< d/2$ we have $K^{\sigma}\in L^r(\R)$ if and only if $r<d/(d-2\sigma)$.

It was proved in \cite{BT1} that $\mathcal{K}^{\sigma}(x,y)$ is controlled by $K^{\sigma}(x-y)$.
To make this section self-contained we include below a short proof of this result.
An estimate of the integral
\begin{equation*}
 E_a(T)=\int_0^1\zeta^{-a}\exp(-T\zeta^{-1})\,d\zeta, \qquad T>0,
\end{equation*}
is needed. The statement below is a refinement of
\cite[Lemma 1.1]{ST1}, see also \cite[Lemma 2.3]{NS2}.
\begin{lemma}
\label{lem:le}
Let $a\in\mathbb R$ be fixed. Then
\begin{equation} \label{estea}
 E_a(T)\lesssim \exp(-T/2), \qquad T \ge 1,
\end{equation}
and for $0<T<1$
\begin{equation*}
 E_a(T)\simeq  
\begin{cases}
1, & \quad a<1, \\
\log (2\slash T), & \quad a=1, \\
T^{-a+1}, & \quad a>1.
\end{cases}
\end{equation*}
\end{lemma}
\begin{proof}
A change of the variable of the integration yields
\begin{equation}\label{int} 
 E_a(T)=T^{-a+1}\int_T^\infty y^{a-2}\exp(-y)\,dy.
\end{equation}
Now the estimate for $T\ge1$ follows since
\begin{equation*}
T^{-a+1}\int_T^\infty y^{a-2}\exp(-y)\,dy\lesssim T^{-a+1}\exp(-3T/4)\lesssim \exp(-T/2).
\end{equation*}
Notice that \eqref{estea} can be improved; in fact we have 
$E_a(T)\lesssim \exp(-T\slash(1+\varepsilon))$ for any fixed $\varepsilon > 0$.

The estimates for $0<T<1$ are verified by splitting the integration in \eqref{int} onto the
intervals $(T,1)$ and $(1,\infty)$. 
Then in the first resulting integral the exponential factor can be neglected, and the second integral
is just a positive constant. This easily implies the desired bounds from above and below.
\end{proof}
\begin{propo}[{\cite[Proposition 2]{BT1}}] \label{pro:comp}
For each $\sigma>0$,
$$
0<\mathcal{K}^{\sigma}(x,y)\lesssim K^{\sigma}(x-y).
$$
\end{propo}
\begin{proof}
The lower estimate is a consequence of the strict positivity of the kernel $G_t(x,y)$.
To show the upper estimate we write
\begin{equation*}
\Gamma(\sigma)\mathcal{K}^{\sigma}(x,y)=\int_0^1 G_t(x,y) t^{\sigma-1}\,dt+\int_1^\infty G_t(x,y)
 t^{\sigma-1}\,dt\equiv \mathcal{J}^{\sigma}_0(x,y)+\mathcal{J}^{\sigma}_\infty(x,y).
\end{equation*}
Then
\begin{equation*}
\mathcal{J}^{\sigma}_\infty(x,y)\lesssim \int_1^\infty e^{-dt}\exp\bigg(-\frac14{\|x-y\|^2}\bigg)
 t^{\sigma-1}\,dt\lesssim 
\exp\bigg(-\frac{\|x-y\|^2}{4}\bigg)
\end{equation*}
and
\begin{equation*}
\mathcal{J}^{\sigma}_0(x,y)\lesssim \int_0^1\exp\bigg(-\frac14\frac{\|x-y\|^2}{t}\bigg) 
	t^{\sigma-d/2-1}\,dt.
\end{equation*}
To treat the last integral we use Lemma \ref{lem:le}
and then combine the obtained bounds of $\mathcal{J}^{\sigma}_0(x,y)$ and 
$\mathcal{J}^{\sigma}_\infty(x,y)$. The required estimate of $\mathcal{K}^{\sigma}(x,y)$ follows.
\end{proof}
Consider the potential operator  $\mathcal{I}^{\sigma}$,
\begin{equation*}
\mathcal{I}^\sigma f(x)=\int_{\R}\mathcal{K}^{\sigma}(x,y)f(y)\,dy,
\end{equation*}
defined on the natural domain $\domain \mathcal{I}^\sigma$ consisting of those functions $f$ 
for which the above integral is convergent $x$-a.e. 
(heuristically, $\mathcal{I}^{\sigma}f = \frac{1}{\Gamma(\sigma)} \int_0^{\infty}
e^{-t\mathcal{H}}f\, t^{\sigma-1}dt$).
By Proposition \ref{pro:comp} and the fact that
$K^\sigma\in L^1(\R)$ we see that  $L^p(\R)\subset \domain \mathcal{I}^\sigma$, $1\le p\le\infty$.

The following result has recently been proved by Bongioanni and Torrea. 
Here we include in addition a discussion of the case $\sigma\ge d\slash 2$.
\begin{theor}[{\cite[Theorem 8]{BT1}}] \label{thm:Her}
Let $\sigma>0$ and $1\le p\le\infty$, $1\le q\le\infty$. If $\sigma\ge d/2$, then
\begin{equation}\label{bound}
\|\mathcal{I}^\sigma f\|_q\lesssim \|f\|_p, \qquad f\in L^p(\R),
\end{equation}
excluding the cases when $\sigma=d\slash 2$ and either
$p=\infty$, $q=1$ or $p=1$, $q=\infty$.
If $0<\sigma< d/2$, then \eqref{bound} holds if 
$\frac1p-\frac{2\sigma}d\le \frac 1q<\frac 1p +\frac{2\sigma}d$, 
with exclusion of the cases: $p=1$ and $q = \frac {d}{d-2\sigma}$ (in which $\mathcal{I}^{\sigma}$
satisfies the weak type $(1,q)$ estimate), and $p=\frac{d}{2\sigma}$ and $q=\infty$. 
Moreover, in each of the cases of strong type $(p,q)$, $p < \infty$, for any $ k\in \N$ we have
\begin{equation}\label{co}
\langle\mathcal{I}^\sigma f, h_k\rangle=\lambda_k^{-\sigma}\langle f, h_k\rangle, 
\qquad f\in L^p(\R).
\end{equation}
\end{theor}
\begin{proof}
Consider first the case $\sigma\ge d\slash 2$. If  $1\le p\le q\le\infty$,
since $0<\mathcal{K}^{\sigma}(x,y)\lesssim K^\sigma(x-y)$, the proof of \eqref{bound}
reduces to checking a similar estimate with $\mathcal{I}^\sigma$ replaced by the convolution 
operator $T^\sigma\colon f\mapsto K^\sigma \ast f$.
Recall that classical Young's inequality has the form
$$
\|g\ast f\|_q\le \|g\|_r\|f\|_p, \qquad \frac 1p+\frac1r=1+\frac1q, \quad 1\le p,q,r\le\infty
$$
(in particular it follows that if $g\in L^r$ and $f\in L^p$, then
$g \ast f(x)$ is well defined $x$-a.e.). Taking $g=K^\sigma$ above shows that 
$T^\sigma\colon L^p(\R)\to L^q(\R)$ boundedly provided $\sigma > d/2$ 
and $1\le p\le q\le\infty$; for $\sigma=d\slash 2$ the case $p=1$, $q=\infty$, is excluded. 
If $q<p$, then we argue as in the proof of \cite[Theorem 8,\,(iii)]{BT1} to get 
$\|\mathcal{I}^\sigma f\|_q\lesssim \|f\|_\infty$ for $1\le q<\infty$ with exclusion of $q=1$ 
when $\sigma=d\slash 2$. Similarly, we proceed as in the proof of \cite[Theorem 8,\,(iv)]{BT1} 
to get $\|\mathcal{I}^\sigma f\|_1\lesssim \|f\|_p$ for $p<\infty$. 
Then \eqref{bound} follows for $1<q<p<\infty$ by interpolation.

If $0<\sigma< d/2$ and $\frac{1}{q} > \frac{1}{p}-\frac{2\sigma}{d}$, 
then using Young's inequality is limited to $\frac 1q\le \frac 1p$.
In the endpoint case when $\frac{1}{q} = \frac{1}{p}-\frac{2\sigma}{d}$
the desired conclusion follows from Theorem \ref{thm:HLS}. This is because (see \cite[(2.9)]{ST1}) 
$$
G_t(x,y)\le W_t(x-y),
$$
where $W_t$ denotes the Gauss-Weierstrass kernel in $\R$, which implies
$\mathcal{I}^\sigma f\lesssim I^{2\sigma} f$ for any nonnegative $f$.
The case $\frac1p< \frac 1q<\frac 1p +\frac{2\sigma}d$ is more delicate and requires further 
arguments based on the estimate
$$
\|x\|^{2\sigma}\int_{\R}\mathcal{K}^{\sigma}(x,y)\,dy\le C,\qquad x\in\R,
$$
and an interpolation argument; we refer to \cite{BT1} for details.

To verify \eqref{co} observe that for each fixed $k\in \N$ the mapping 
$f\mapsto \langle\mathcal{I}^\sigma f, h_k\rangle$ is a bounded linear functional on 
$L^p(\R)$. This is because
$$
|\langle\mathcal{I}^\sigma f, h_k\rangle|\le \|\mathcal{I}^\sigma f\|_q\|h_k\|_{q'}
\lesssim \|h_k\|_{q'}\|f\|_p.
$$
Moreover, in the proof of Corollary \ref{negative} below it is checked that this functional agrees, on the 
dense in $L^p(\R)$ linear span of Hermite functions, with the linear functional
$f\mapsto \lambda_k^{-\sigma}\langle\ f, h_k\rangle$, which is also bounded on  
$L^p(\R)$. Hence both functionals coincide and \eqref{co} is justified.
\end{proof}
It is worth mentioning that in the case $0<\sigma<d\slash 2$ of the above theorem,
in some occurrences the constraint between $p$ and $q$  gives optimal ranges of $p$ and $q$ 
for \eqref{bound} to hold. This happens when $p=1$ or $p=\infty$ (then $1\le q<\frac d{d-2\sigma}$ 
or $\frac d{2\sigma}<q\le\infty$ are optimal, respectively), and $q=1$ or $q=\infty$ (then 
$1\le p<\frac d{d-2\sigma}$ or $\frac d{2\sigma}<p\le\infty$ are  sharp, respectively).
The corresponding proofs can be found in \cite{BT1}.
\begin{corollary} \label{negative}
Let $\sigma>0$ and $(p,q)$, $1\le p<\infty$, $1\le q \le\infty$, be such a pair that \eqref{bound} holds. 
Then $\mathcal{H}^{-\sigma}$ extends to a bounded operator from 
$L^p(\R)$ to $L^q(\R)$. 
Moreover, denoting this extension by $\mathcal{H}^{-\sigma}_{pq}$, 
in each of the cases, for any $k\in\N$ we have
\begin{equation} \label{coeff}
\langle\mathcal{H}^{-\sigma}_{pq} f, h_k\rangle=\lambda_k^{-\sigma}\langle f, h_k\rangle, 
\qquad f\in L^p(\R).
\end{equation}
\end{corollary}
\begin{proof}
In view of Theorem \ref{thm:Her} it suffices to show
that $\mathcal{H}^{-\sigma}=\mathcal{I}^\sigma$ as operators on $L^2(\R)$.
This follows by observing that both operators,
being bounded on $L^2(\R)$,  coincide  on the dense in $L^2(\R)$ linear span of Hermite functions. 
Indeed, to check that  $\mathcal{H}^{-\sigma}h_k=\mathcal{I}^\sigma h_k$, $k\in \N$, we write
\begin{align*}
\Gamma(\sigma)\int_{\R}\mathcal{K}^{\sigma}(x,y)h_k(y)\,dy&=\int_{\R}\int_0^\infty 
G_t(x,y)t^{\sigma-1}\,dt\, h_k(y)\,dy\\
&=\int_0^\infty t^{\sigma-1} e^{-t\mathcal{H}}h_k(x) \, dt \\
&=\int_0^\infty t^{\sigma-1}e^{-t\lambda_k}\,dt\,h_k(x)\\
&=\Gamma(\sigma)\mathcal{H}^{-\sigma}h_k(x).
\end{align*}
Application of Fubini's theorem in the second identity above was possible since, 
for any fixed $x\in\R$,
$$
\int_{\R}\int_0^\infty  G_t(x,y)t^{\sigma-1}
 |h_k(y)|\,dt\,dy=\int_{\R}\mathcal{K}^{\sigma}(x,y)|h_k(y)|\,dy<\infty;
$$
this is because $\mathcal{K}^{\sigma}(x,\cdot)\in L^1(\R)$ for any fixed $x\in\R$, 
and $h_k\in L^\infty(\R)$. 

Considering \eqref{coeff}, given $1\le p<\infty$, the subspace $L^2(\R)\cap L^p(\R)$ 
is dense in $ L^p(\R)$, hence the extension
$\mathcal{H}^{-\sigma}_{pq}$ coincides with $\mathcal{I}^\sigma$ as a bounded operator from 
$L^p(\R)$ to $L^q(\R)$. Thus \eqref{coeff} follows from \eqref{co}.
\end{proof}

It is worth to point out that for $1<q<\infty$ the assertion of Corollary \ref{negative} remains valid 
if in \eqref{neg}, the definition of $\mathcal{H}^{-\sigma}$, the
multi-sequence $\{(2|k|+d)^{-\sigma}\}$
is replaced by another multi-sequence of similar smoothness, for instance by
$\{(|k|+1)^{-\sigma}\}$ (it would be reasonable to refer to the resulting operator as to the
{fractional integral operator} for Hermite function expansions; 
then accordingly $\lambda_k$ in \eqref{coeff} must be replaced by $(|k|+1)$). 
Indeed, this is a simple consequence of a multiplier theorem for
multi-dimensional Hermite function expansions, see \cite[Theorem 4.2.1]{Th} or 
\cite[Theorems 7.10-11]{DOS}, since the multiplier multi-sequence  
$\{(\frac{2|k|+d}{|k|+1})^\sigma\}$ defines a bounded operator on $L^q(\R)$ for each $1<q<\infty$.

Theorem \ref{thm:Her} extends the result of Gasper and Trebels \cite[Theorem 3]{GT} in several directions.
First of all, the result is multi-dimensional. Secondly, the restriction $\sigma<1/2$ 
(in the case $d=1$ discussed in \cite{GT})
is released. Finally, the constraint $\frac1q=\frac1p-2\sigma$ (still in the case $d=1$) occurs to be
unnecessary, and in addition the case $p=1$ is admitted.

We take an opportunity to generalize Theorem \ref{thm:Her}, and below we give a 
two-weight extension of Theorem \ref{thm:Her} in the spirit of the result by Stein and Weiss 
stated in Theorem \ref{thm:StWe}. It is clear that the range of $q$ that depends on $p$ in 
Theorem \ref{thm:Hwt} below is not optimal.
\begin{theor} \label{thm:Hwt}
Let $\sigma > 0$, $1<p \le q < \infty$, $a < d\slash p'$, $b < d\slash q$, $a+b \ge 0$.
\begin{itemize}
\item[(i)]
If $\sigma \ge d\slash 2$, then $\mathcal{I}^{\sigma}$ maps boundedly $L^p(\R,\|x\|^{a p})$ into
$L^q(\R,\|x\|^{-b q})$. 
\item[(ii)]
If $\sigma < d\slash 2$, then the same boundedness holds under the additional condition
\begin{equation} \label{cnd7}
\frac{1}{q} \ge \frac{1}{p} - \frac{2\sigma-a-b}{d}.
\end{equation}
\end{itemize}
Moreover, under the assumptions ensuring boundedness of $\mathcal{I}^{\sigma}$ from $L^p(\R,\|x\|^{a p})$
into $L^q(\R,\|x\|^{-b q})$,
\begin{equation}\label{co2}
\langle\mathcal{I}^\sigma f, h_k\rangle=\lambda_k^{-\sigma}\langle f, h_k\rangle, 
\qquad f\in L^p(\R, \|x\|^{a p}).
\end{equation}
\end{theor}

Note that implicitly Theorem \ref{thm:Hwt} asserts the inclusion
$L^p(\R,\|x\|^{a p}) \subset \domain \mathcal{I}^{\sigma}$
in all the cases when weighted $L^p-L^q$ boundedness holds. 
The proof of Theorem \ref{thm:Hwt} requires suitable weighted inequalities for convolutions.
We shall use those obtained by Kerman \cite{K}, which we formulate below for an easy reference.
\begin{lemma}[{\cite[Theorem 3.1]{K}}] \label{lem:conv}
Assume that the parameters $p,q,r,a,b,\eta$ satisfy
\begin{equation} \label{CK1}
1<p,q,r<\infty, \qquad \frac{1}{q} \le \frac{1}{p}+\frac{1}{r},
\end{equation}
\begin{equation} \label{CK2}
\frac{1}{q}-\frac{1}{p} - \bigg(\frac{a+b}{d}-1\bigg) = \frac{1}{r} + \frac{\eta}{d},
\end{equation}
\begin{equation} \label{CK3}
a < \frac{d}{p'}, \qquad b < \frac{d}{q}, \qquad \eta < \frac{d}{r'},
\end{equation}
\begin{equation} \label{CK4}
a+b \ge 0, \qquad a+ \eta \ge 0, \qquad b+\eta \ge 0.
\end{equation}
If $g \in L^r(\|x\|^{\eta r})$ and $f \in L^p(\|x\|^{ap})$, then $g * f(x)$ is well defined for
a.e. $x \in \R$ and
$$
\|g * f\|_{L^q(\|x\|^{-bq})} \le C \|g\|_{L^r(\|x\|^{\eta r})} \|f\|_{L^p(\|x\|^{ap})}
$$
with a constant $C$ independent of $g$ and $f$.
\end{lemma}
\begin{proof}[Proof of Theorem \ref{thm:Hwt}]
We first deal with case (i), that is when $\sigma \ge d\slash 2$.
As in the proof of Theorem \ref{thm:Her}, 
it is enough to prove the statement with $\mathcal{I}^{\sigma}$ replaced by the convolution
operator $T^{\sigma} \colon f \mapsto K^{\sigma} * f$. To get the desired boundedness of $T^{\sigma}$
we will apply Lemma \ref{lem:conv} with a suitable choice of $r$ and $\eta$, so that, 
in particular, assumptions \eqref{CK1}-\eqref{CK4} are satisfied. 

It is easy to check that the kernel $K^{\sigma}$ is in
$L^r(\|x\|^{\eta r})$ if and only if the right-hand side in \eqref{CK2} is positive.
Notice that \eqref{CK1} is satisfied with any $1<r<\infty$, since $p \le q$. 
Also, the first two inequalities of \eqref{CK3} hold by the assumptions, and together they
imply that the left-hand side in \eqref{CK2}
$$
\xi := \frac{1}{q}-\frac{1}{p} - \bigg(\frac{a+b}{d}-1\bigg) > 0.
$$
On the other hand, since by assumption $a+b \ge 0$, 
the first inequality in \eqref{CK4} holds and it follows
that the quantity $\xi$ is in the interval $(0,1)$ except for the singular case when $p=q$ and $a+b=0$,
which will be treated in a moment separately.
Finally, the third inequality in \eqref{CK3} is equivalent to saying that the right-hand side in
\eqref{CK2} is less than $1$.
To make use of Lemma \ref{lem:conv} it remains to show that any admissible value of $\xi$
can be attained by the right-hand side in \eqref{CK2}, with $1<r<\infty$ and $\eta$ such that
$a+\eta \ge 0$ and $b+\eta \ge 0$.
If $a,b>0$, then we simply take $\eta=0$ and let $r=1\slash \xi$.
When $a\le 0$ we take $\eta=-a$ and we have
$$
1 > \frac{1}{r} = \xi+\frac{a}{d} = \frac{1}{q}-\frac{1}{p}-\frac{b}{d}+1 > \frac{1}{q}-\frac{1}{p}
-\frac{1}{q}+1 = \frac{1}{p'}>0,
$$
so the appropriate choice of $r$ is again possible.
Finally, if $b \le 0$, then we take $\eta=-b$ and can choose suitable $r$ since now
$$
1 > \frac{1}{r} = \xi+\frac{b}{d} = \frac{1}{q}-\frac{1}{p}-\frac{a}{d}+1 > \frac{1}{q}-\frac{1}{p}
-\frac{1}{p'}+1 = \frac{1}{q}>0.
$$
We see that Lemma \ref{lem:conv} does the job except for the singular case distinguished above.

To cover the case when $p=q$ and $a+b=0$ we observe that $T^{\sigma}$ is controlled by the (centered)
Hardy-Littlewood maximal operator $M$. This is because the convolution kernel is integrable, radial,
and essentially radially decreasing, see \cite[Proposition 2.7]{Duo}.
But for $1<p<\infty$, $M$ is bounded on $L^p(w)$ with any weight $w$ in Muckenhoupt's class $A_p$.
The conclusion follows by the easy to verify fact that a power weight $w(x)=\|x\|^{ap}$ belongs to $A_p$
if and only if $-d\slash p<a<-d\slash p'$. This completes proving case (i).

We now treat case (ii), when $0<\sigma < d\slash 2$.
Since $\mathcal{I}^{\sigma}$ is dominated by a constant times $I^{2\sigma}$, 
see the proof of Theorem \ref{thm:Her},
the case of equality in \eqref{cnd7} is covered by Theorem \ref{thm:StWe}. Therefore we may assume
that
$$
\xi > 1-\frac{2\sigma}{d}.
$$ 
But the kernel $K^{\sigma}$ belongs to $L^r(\|x\|^{\eta r})$ 
if and only if the right-hand side in \eqref{CK2}
is greater than $1-2\sigma\slash d$. In this position we repeat the reasoning of case (i).

For the proof of \eqref{co2} we copy the argument leading to the proof of \eqref{co}. 
One only has to know that $\|h_k\|_{L^{q'}(\|x\|^{bq'})}<\infty$, but this holds in view of the
inequality $b>-d/q'$ following from the assumptions imposed on $d,p,q,a,b$. 
\end{proof}

\section{Laguerre function expansions of Hermite type} \label{sec:lag_herm}
Let $k=(k_1,\ldots,k_d) \in  \N$ and
$\alpha = (\alpha _1, \ldots , \alpha _d) \in (-1,\infty)^{d}$
be multi-indices.
The Laguerre function $\varphi _{k}^{\alpha}$ on $ \R_{+}$ is the tensor product
$$
\varphi _{k}^{\alpha}(x) = \varphi _{k_1}^{\alpha _1}(x_1) \cdot \ldots \cdot
\varphi _{k_d}^{\alpha _d}(x_d), \qquad x = (x_1, \ldots ,x_d)\in \R_{+},
$$
where $\varphi _{k_i}^{\alpha _i}$ are the one-dimensional Laguerre functions
$$
\varphi _{k_i}^{\alpha _i}(x_i) =
\left(\frac{2 \Gamma (k_i+1)}{\Gamma(k_i+\alpha _i +1)}\right)
^{1 \slash 2} L_{k_i}^{\alpha _i} (x_i^{2}) x_i^{\alpha _i + 1 \slash 2}
e^{-{x_i^2}/{2}}, \quad \quad x_i > 0, \quad i = 1,\ldots , d;
$$
given $\alpha_i>-1$ and $k_i\in\mathbb N$, $L^{\alpha_i}_{k_i}$
denotes the Laguerre polynomial of degree $k_i$ and order $\alpha_i,$ see \cite[p.\,76]{Leb}.

Each $\varphi _k^{\alpha}$ is an eigenfunction of the differential operator
$$
 L_\alpha^H=
-\Delta +\|x\|^2 + \sum _{i=1}^{d} \frac{1}{x_i^2} \left(\alpha _i^2 -\frac{1}{4}\right),
$$
the corresponding eigenvalue being $\lambda_k^\alpha=4|k|+2|\alpha |+2d$, that is
$
L_\alpha^H\varphi _k^{\alpha}=\lambda_k^\alpha\varphi _k^{\alpha};
$
here by $|\alpha|$ we mean $|\alpha|=\alpha_1+\ldots+\alpha_d$ (thus $|\alpha|$ may be negative).
The operator $L_\alpha^H$ is symmetric and positive in $L^2(\R_{+})$, and
the system $\{ \varphi _{k}^{\alpha} : k \in \N \}$ is an orthonormal basis in $L^2(\R_{+}).$

As defined in \cite[p.\,402]{NS2}, $L_\alpha^H$ has a self-adjoint extension 
$\mathcal{L}_\alpha^H$ whose spectral decomposition is given by the $\varphi_k^{\alpha}$ and
$\lambda_k^{\alpha}$.
The heat-diffusion semigroup $\{e^{-t\mathcal{L}_\alpha^H}\}_{t>0}$ 
generated by $\mathcal{L}_{\alpha}^H$,
\begin{equation*}
e^{-t\mathcal{L}_\alpha^H}f=\sum_{n=0}^{\infty} e^{-t(4n+2|\alpha|+2d)} \sum_{|k|=n}
\langle f,\varphi _k^{\alpha} \rangle\varphi_k^{\alpha}, \qquad f\in L^2(\R_+),
\end{equation*}
is a strongly continuous semigroup of contractions on $L^2(\R_+)$.
We have the integral representation
\begin{equation*} 
  e^{-t\mathcal{L}_\alpha^H}f(x)=\int_{\R_{+}} G_t^{\alpha,H}(x,y)f(y)\,dy, 
  \qquad x\in \R_{+},
\end{equation*}
where
\begin{equation*}
G^{\alpha,H}_t(x,y)=\sum_{n=0}^\infty e^{-t(4n+2|\alpha|+2d)} 
\sum_{|k|=n} \varphi_k^\alpha(x)\varphi_k^\alpha(y),\qquad x,y\in \R_{+}.
\end{equation*}
It is known, cf. \cite[(4.17.6)]{Leb}, that
$$
G^{\alpha,H}_t(x,y)=(\sinh 2t)^{-d}\exp\Big({-\frac{1}{2} \coth(2t)\big(\|x\|^{2}+\|y\|^{2}\big)}\Big)
\prod^{d}_{i=1} \sqrt{x_i y_i}\, I_{\alpha_i}\left(\frac{x_i y_i}{\sinh 2t}\right).
$$
Here $I_\nu$ denotes the modified Bessel function of the first kind and order $\nu$;
considered on the positive half-line, it is real, positive and smooth for any $\nu > -1$.

It was observed in \cite{NS2} (see the proof of Proposition 2.1 there) 
that given $\alpha \in [-1\slash 2, \infty)^d,$ there exists a constant $C_{\alpha}$ such that
\begin{equation}\label{cc}
G^{\alpha,H}_t(x,y) \le C_{\alpha} G^{\alpha_o,H}_t(x,y), \qquad t>0, \quad x,y \in \R_{+},
\end{equation}
with $\alpha_o = (-1\slash 2, \ldots, -1\slash 2)$. This was based on
the asymptotics, cf. \cite[(5.16.4),\,(5.16.5)]{Leb},
\begin{equation}\label{bes}
I_\nu(z)\simeq z^\nu,\quad z\to 0^+;\qquad I_\nu(z)\simeq z^{-1/2}e^z, \quad z\to \infty
\end{equation}
(more information on $C_{\alpha}$ can be obtained from properties of the function 
$\nu \mapsto I_{\nu}(x)$, $x\in\mathbb R_+$, see \cite{NS2}). 
Moreover, we have (cf. \cite[(A.2)]{NS2})
\begin{align}\label{nowe}
G_t^{\alpha_o,H}(x,y)&=\sum_{\varepsilon\in\mathcal{E}}G_t(\varepsilon x,y), \qquad x,y\in\R_+, 
\end{align}
where $\mathcal{E} =\{(\varepsilon_1,\ldots,\varepsilon_d): \varepsilon_i=\pm1\}$
and $\varepsilon x=(\varepsilon_1 x_1,\ldots, \varepsilon_d x_d)$.

Given $\sigma>0$, consider the operator $(\mathcal{L}_\alpha^H)^{-\sigma}$ defined on $L^2(\R_+)$ by the 
spectral series
\begin{equation}\label{neg_H}
(\mathcal{L}_\alpha^H)^{-\sigma}f= \sum_{k \in \mathbb{N}^d} 
(\lambda_k^\alpha)^{-\sigma}	\langle f,\varphi^\alpha_k\rangle \,\varphi^\alpha_k.
\end{equation}
Observe that $(\mathcal{L}_\alpha^H)^{-\sigma}$ is a contraction on $L^2(\R_+)$
if $\alpha \in [-1\slash 2,\infty)^d$.

We next define the potential kernel 
\begin{equation}\label{ker_H}
\mathcal{K}^{\alpha,\sigma}_H(x,y)= 
\frac1{\Gamma(\sigma)}\int_0^\infty G^{\alpha,H}_t(x,y) t^{\sigma-1}\,dt, \qquad x,y\in\R_{+},
\end{equation}
and the potential operator 
\begin{equation}\label{aa_H}
\mathcal{I}^{\alpha,\sigma}_H f(x)=\int_{\R_+}\mathcal{K}^{\alpha,\sigma}_H(x,y)f(y)\,dy, \qquad x\in\R_{+},
\end{equation}
and note that \eqref{cc} and  \eqref{nowe} lead immediately to 
\begin{equation} \label{dd_H}
\mathcal{K}^{\alpha,\sigma}_H(x,y)\le C_\alpha \mathcal{K}^{\alpha_o,\sigma}_H(x,y)
= C_{\alpha} \sum_{\varepsilon\in\mathcal{E}}\mathcal{K}^\sigma(\varepsilon x,y), \qquad x,y\in\R_+.
\end{equation}
Consequently, 
\begin{equation}\label{ee_H}
\mathcal{I}^{\alpha,\sigma}_Hf(x)\le C_{\alpha} \sum_{\varepsilon\in\mathcal{E}}\mathcal{I}^\sigma 
f(\varepsilon x), \qquad x\in\R_+,
\end{equation}
for any $f\ge0$ defined on $\R_+$, where on the right-hand side of \eqref{ee_H} the function $f$ is 
understood as the extension of $f$ onto $\R$ obtained by setting 0 outside $\R_+$. Note at this point that 
the integral in \eqref{ker_H} is convergent for every $x,y\in\R_+$ when  $\sigma > d/2$, while 
for $0<\sigma\le d/2$ the integral converges provided $x\neq y$ (this is a consequence of the former 
statement concerning convergence of the integral in \eqref{ker}, or it can be seen directly from the 
decay of $G_t^{\alpha,H}(x,y)$ when $t\to\infty$ and $t\to0^+$). Similarly, $L^p(\R_+)\subset 
\domain\mathcal{I}^{\alpha,\sigma}_H$, $1\le p\le\infty$, due to \eqref{ee_H}
and the former statement concerning $\domain \mathcal{I}^{\sigma}$; 
here $\domain \mathcal{I}^{\alpha,\sigma}_H$ denotes
the natural domain of $\mathcal{I}^{\alpha,\sigma}_H$ consisting of those functions $f$ 
for which the integral in \eqref{aa_H} is convergent $x$-a.e. 

Thus, in view of Theorem \ref{thm:Her}, we obtain the following result.
\begin{theor} \label{thm:lag_H}
Let $\alpha \in [-1/2,\infty)^{d}$ and $\sigma>0$. 
If $1\le p\le\infty$, $1\le q \le\infty$, and $(p,q)$ is one of the pairs specified in
Theorem \ref{thm:Her} for which \eqref{bound} holds, then 
\begin{equation}\label{bound2}
\|\mathcal{I}^{\alpha,\sigma}_H f\|_q\lesssim \|f\|_p, \qquad f\in L^p(\R_+),
\end{equation}
is also true. Moreover, for each such pair $(p,q)$ with $p<\infty$,
for any $k\in \N$ we have
\begin{equation}\label{co3}
\langle\mathcal{I}^{\alpha,\sigma}_H f, \varphi^\alpha_k\rangle=(\lambda_k^\alpha)^{-\sigma}\langle f,
\varphi^\alpha_k\rangle, \qquad f\in L^p(\R_+).
\end{equation}
\end{theor}
\begin{proof}
Only \eqref{co3} requires an explanation. 
It is proved in the same way as \eqref{co} was established. 
The important fact to be used is that $\|\varphi^\alpha_k\|_{q'}<\infty$ for $1\le q'\le \infty$, 
and this is indeed assured by the assumption imposed on $\alpha$.
\end{proof}
Furthermore, we also obtain.
\begin{corollary} \label{cor:neg2}
Let  $\alpha \in [-1/2,\infty)^{d}$, $\sigma>0$ and $(p,q)$, $p<\infty$, be such a pair 
that \eqref{bound2} holds.
Then $(\mathcal{L}_\alpha^H)^{-\sigma}$ extends to a bounded operator from 
$L^p(\R_+)$ to $L^q(\R_+)$. 
Moreover, denoting this extension by $(\mathcal{L}_\alpha^H)_{pq}^{-\sigma}$, 
in each of the cases, for any $k\in\N$ we have
\begin{equation} \label{coeff2}
\langle(\mathcal{L}_\alpha^H)_{pq}^{-\sigma} f, \varphi^\alpha_k\rangle=
(\lambda^\alpha_k)^{-\sigma}\langle f, \varphi^\alpha_k\rangle, 
\qquad f\in L^p(\R_+).
\end{equation}
\end{corollary}
\begin{proof}
The arguments are parallel to those used in the proof of Corollary \ref{negative}.
Application of Fubini's theorem in the relevant place is possible since,
for any fixed $x\in\R_+$,
$$
\int_{\R_+}\int_0^\infty  G_t^{\alpha,H}(x,y)t^{\sigma-1}
|\varphi^\alpha_k(y)|\,dt\,dy=\int_{\R_+}\mathcal{K}^{\alpha,\sigma}_H(x,y)|\varphi^\alpha_k(y)|\,dy<\infty;
$$
this is because $\mathcal{K}^{\alpha,\sigma}_H(x,\cdot)\in L^1(\R_+)$ for any $x\in\R_+$ fixed,
and $\varphi^\alpha_k\in L^\infty(\R_+)$. 
\end{proof}

A weighted analogue of Theorem \ref{thm:lag_H} is obtained by
combining \eqref{ee_H} with Theorem \ref{thm:Hwt}.
\begin{theor} \label{thm:lag_Hw}
Let $\alpha \in [-1/2,\infty)^{d}$, $\sigma > 0$, $1<p \le q < \infty$, $a < d\slash p'$, 
$b < d\slash q$, $a+b \ge 0$.
\begin{itemize}
\item[(i)]
If $\sigma \ge d\slash 2$, then $\mathcal{I}_H^{\alpha,\sigma}$ maps boundedly $L^p(\R_+,\|x\|^{a p})$ into
$L^q(\R_+,\|x\|^{-b q})$. 
\item[(ii)]
If $\sigma < d\slash 2$, then the same boundedness holds under the additional condition \eqref{cnd7}.
\end{itemize}
Moreover, under the assumptions ensuring boundedness of $\mathcal{I}^{\alpha,\sigma}_H$
from $L^p(\R_+,\|x\|^{a p})$ into $L^q(\R_+,\|x\|^{-b q})$,
\begin{equation*}
\langle\mathcal{I}^{\alpha,\sigma}_H f, \varphi^\alpha_k\rangle
=(\lambda^{\alpha}_k)^{-\sigma}\langle f, \varphi^\alpha_k\rangle, \qquad f\in L^p(\R_+, \|x\|^{a p}).
\end{equation*}
\end{theor}
Note that implicitly Theorem \ref{thm:lag_Hw} asserts the inclusion
$L^p(\R_+,\|x\|^{a p}) \subset \domain \mathcal{I}^{\alpha,\sigma}_H$
in all the cases when weighted $L^p-L^q$ boundedness holds. 

Similarly as in the Hermite setting, we point out that for $1<q<\infty$ the assertion of 
Corollary \ref{cor:neg2} remains valid if in \eqref{neg_H}, the definition of 
$\mathcal{L}_H^\alpha$, the multi-sequence
$\{(\lambda_k^{\alpha})^{-\sigma}\}$ is replaced by $\{(|k|+1)^{-\sigma}\}$ 
(then accordingly $\lambda_k^\alpha$ in \eqref{coeff2} must be replaced by $(|k|+1)$)
or by another sufficiently smooth multi-sequence. Again this is a simple consequence of a 
multiplier theorem, this time for Laguerre expansions of Hermite type, see 
\cite[Theorem 6.4.3]{Th} or  \cite[Theorem 7.12]{DOS}, since 
the multiplier multi-sequence $\{(\frac{4|k|+2|\alpha|+2d}{|k|+1})^\sigma\}$ 
generates a bounded operator on each $L^q(\R_+)$, $1<q<\infty$.

\section{Laguerre function expansions of convolution type} \label{sec:lag}
In this section we shall work on the space $\R_{+}$, $d\ge 1$, equipped with the measure
$$
\mu_{\alpha}(dx) = x_1^{2\alpha_1+1} \cdot \ldots \cdot x_d^{2\alpha_d+1} \, dx.
$$
Given multi-indices $k \in \N$ and $\alpha \in (-1,\infty)^{d}$,
the Laguerre functions $\ell^{\alpha}_k$ are
$$
\ell_{k}^{\alpha}(x) = \ell _{k_1}^{\alpha _1}(x_1) \cdot \ldots \cdot
\ell_{k_d}^{\alpha _d}(x_d), \qquad x = (x_1, \ldots ,x_d)\in \R_{+},
$$
where $\ell_{k_i}^{\alpha _i}$ are the one-dimensional Laguerre functions
$$
\ell_{k_i}^{\alpha _i}(x_i) =
\left(\frac{2 \Gamma (k_i+1)}{\Gamma(k_i+\alpha _i +1)}\right)
^{1 \slash 2} L_{k_i}^{\alpha _i} (x_i^{2})e^{-{x_i^2}/{2}}, \quad \quad x_i > 0,
\quad i = 1,\ldots , d.
$$

Each $\ell_k^{\alpha}$ is an eigenfunction of the differential operator 
$$
L_\alpha=-\Delta +\|x\|^2 - \sum _{i=1}^{d} \frac{2\alpha_i+1}{x_i} \frac{\partial}{\partial x_i}
$$
with the corresponding eigenvalue $\lambda_k^\alpha=4|k|+2|\alpha |+2d$, that is
$
L_\alpha\ell_k^{\alpha}=\lambda_k^\alpha\ell_k^{\alpha}.
$
The operator $L_\alpha$ is symmetric and positive in $L^2(d\mu_{\alpha})$,
and the system $\{ \ell_{k}^{\alpha} : k \in \N \}$ is an orthonormal
basis in $L^2(d\mu_{\alpha}).$

Let $\mathcal{L}_\alpha $ denote the self-adjoint extension of $L_\alpha$ as defined in 
\cite[p.\,646]{NS1}, whose spectral decomposition is given by the $\ell_k^{\alpha}$ 
and $\lambda_k^{\alpha}$.
The heat-diffusion semigroup $\{e^{-t\mathcal{L}_{\alpha}}:t> 0\}$ generated by
$\mathcal{L}_{\alpha}$ is a strongly continuous semigroup of contractions on
$L^2(d\mu_{\alpha})$. By the spectral theorem,
\begin{equation*}
e^{-t\mathcal{L}_{\alpha}} f=\sum_{n=0}^\infty e^{-t(4n+2|\alpha|+2d)} \sum_{|k|=n}
\langle f,\ell_k^{\alpha} \rangle_{d\mu_{\alpha}}\ell_k^{\alpha}, \qquad f\in L^2(d\mu_{\alpha}).
\end{equation*}
We have the integral representation
\begin{equation*}
  e^{-t\mathcal{L}_{\alpha}} f(x)=\int_{\R_{+}} G_t^\alpha(x,y)f(y)\,d\mu_{\alpha}(y),
  \qquad x\in \R_{+},
\end{equation*}
where the heat kernel is given by
\begin{equation*}
G^\alpha_t(x,y)=\sum_{n=0}^\infty e^{-t(4n+2|\alpha|+2d)} \sum_{|k|=n}
\ell_k^\alpha(x)\ell_k^\alpha(y).
\end{equation*}

As in the previous settings, we define the potential kernel
\begin{equation} \label{ker2}
\mathcal{K}^{\alpha,\sigma}(x,y)= 
\frac1{\Gamma(\sigma)}\int_0^\infty G^{\alpha}_t(x,y) t^{\sigma-1}\,dt, \qquad x,y\in\R_+,
\end{equation}
and the potential operator 
\begin{equation*}
\mathcal{I}^{\alpha,\sigma} f(x)=\int_{\R_+}\mathcal{K}^{\alpha,\sigma}(x,y)f(y)\,d\mu_\alpha(y).
\end{equation*}

Let $S_{\alpha}$ be the multiplication
operator
$$
S_{\alpha}f(x) = f(x) \prod_{i=1}^d x_i^{\alpha_i+1\slash 2}, \qquad x \in \R_{+}.
$$
Then $S_{\alpha}$ is an isometric isomorphism of $L^2(\R_{+},d\mu_{\alpha})$ onto
$L^2(\R_{+},dx)$, which intertwines the differential operators
$L_{\alpha}^H$ and $L_{\alpha}$,
$$
L_{\alpha}^H \circ S_{\alpha} = S_{\alpha} \circ L_{\alpha}.
$$
Moreover, $\varphi^{\alpha}_k = S_{\alpha} \ell^{\alpha}_k$, $k\in \N$,
and the eigenvalues in both settings coincide. Therefore, 
\begin{align}\label{11}
G^\alpha_t(x,y)& = G^{\alpha,H}_t(x,y)\prod^{d}_{i=1} (x_i y_i)^{-\alpha_i-1/2}, \\ \label{22}
\mathcal{K}^{\alpha,\sigma}(x,y)
& =\mathcal{K}^{\alpha,\sigma}_H(x,y) \prod^{d}_{i=1} (x_i y_i)^{-\alpha_i-1/2}
\end{align}
and hence
\begin{equation}\label{33}
\mathcal{I}^{\alpha,\sigma} f(x)=\Big(\prod^{d}_{i=1}
x_i^{-\alpha_i-1/2}\Big)\mathcal{I}^{\alpha,\sigma}_H\big(S_\alpha f\big)(x).
\end{equation}
It follows from
\eqref{11} that comments concerning convergence of the integral in \eqref{ker2} are exactly the same as
those describing convergence of the integral in \eqref{ker_H}. 
Moreover, \eqref{33} forces the following relation between the natural domains of
$\mathcal{I}^{\alpha,\sigma}$ and 
$\mathcal{I}^{\alpha,\sigma}_H $: $f\in \domain \mathcal{I}^{\alpha,\sigma}$ if and only if
$S_{\alpha}^{-1}f\in \domain \mathcal{I}^{\alpha,\sigma}_H$.
Further, \eqref{33} gives the following interplay between weighted $L^p-L^q$ estimates for
$\mathcal{I}^{\alpha,\sigma}$ and $\mathcal{I}^{\alpha,\sigma}_H$: 
given two weights $U$ and $V$ on $\R_+$, the inequality
$$
\| \mathcal{I}^{\alpha,\sigma}f\|_{L^q(Vd\mu_{\alpha})}\le 
C \|f\|_{L^p(Ud\mu_{\alpha})}, \qquad f\in L^p(Ud\mu_{\alpha}),
$$ 
is equivalent to 
$$
\| \mathcal{I}^{\alpha,\sigma}_Hf\|_{L^q(\widetilde V)}\le C \|f\|_{L^p(\widetilde U)}, 
\qquad f\in L^p(\widetilde U),
$$ 
where 
$$
\widetilde{U}(x) = U(x) \prod^{d}_{i=1} x_i^{(2\alpha_i+1)(1-\frac p2)}, \qquad 
\widetilde{V}(x) = V(x) \prod^{d}_{i=1} x_i^{(2\alpha_i+1)(1-\frac q2)}.
$$
Consequences of this equivalence regarding $L^p-L^q$ estimates are commented in 
Section \ref{sec:final} below.

Our main results for the Laguerre system $\{\ell_k^{\alpha}\}$ read as follows
(notice that now the role of the dimension is played by the quantity $2|\alpha|+2d$).
\begin{theor} \label{thm:mainunw}
Assume that $\alpha \in [-1\slash 2,\infty)^d$. Let $\sigma > 0$ and $1\le p < \infty$, $1\le q < \infty$.
If $\sigma \ge |\alpha|+d$, then
\begin{equation} \label{eq32}
\|\mathcal{I}^{\alpha,\sigma}f\|_{L^q(d\mu_{\alpha})} \lesssim \|f\|_{L^p(d\mu_{\alpha})}, \qquad
f \in L^p(\R_+,d\mu_{\alpha}).
\end{equation}
If $0 < \sigma < |\alpha|+d$, then \eqref{eq32} holds under the additional condition
$$
\frac{1}{p}-\frac{\sigma}{|\alpha|+d} \le \frac{1}{q} < \frac{1}{p} + \frac{\sigma}{|\alpha|+d},
$$
with exclusion of the case when $p=1$ and $q=\frac{|\alpha|+d}{|\alpha|+d-\sigma}$.
Moreover, under the assumptions ensuring \eqref{eq32},
\begin{equation} \label{eq45}
\langle\mathcal{I}^{\alpha,\sigma} f, \ell^\alpha_k\rangle_{d\mu_{\alpha}}
=(\lambda^{\alpha}_k)^{-\sigma}\langle f, \ell^\alpha_k\rangle_{d\mu_{\alpha}}, 
\qquad f\in L^p(\R_+, d\mu_{\alpha}).
\end{equation}
\end{theor}
Considering the weighted setting, we also prove the following.
\begin{theor} \label{thm:main}
Assume that $\alpha \in [-1\slash 2, \infty)^d$. 
Let $\sigma > 0$, $1<p \le q < \infty$ and 
\begin{equation*} 
a < (2|\alpha|+2d)\slash p', \qquad 
b < (2|\alpha|+2d)\slash q, \qquad a+b \ge 0.
\end{equation*}
\begin{itemize}
\item[(i)]
If $\sigma \ge |\alpha|+d$, then $\mathcal{I}^{\alpha,\sigma}$ maps boundedly 
$L^p(\R_+,\|x\|^{a p}d\mu_{\alpha})$ into $L^q(\R_+,\|x\|^{-b q}d\mu_{\alpha})$. 
\item[(ii)]
If $\sigma < |\alpha|+d$, then the same boundedness holds under the additional condition
\begin{equation*} 
\frac{1}{q} \ge \frac{1}{p} - \frac{2\sigma-a-b}{2|\alpha|+2d}.
\end{equation*}
\end{itemize}
Moreover, under the assumptions ensuring boundedness of $\mathcal{I}^{\alpha,\sigma}$ 
from $L^p(\R_+,\|x\|^{a p}d\mu_{\alpha})$ into $L^q(\R_+,\|x\|^{-b q}d\mu_{\alpha})$, 
\begin{equation} \label{co44}
\langle\mathcal{I}^{\alpha,\sigma} f, \ell^\alpha_k\rangle_{d\mu_{\alpha}}
=(\lambda^{\alpha}_k)^{-\sigma}\langle f, \ell^\alpha_k\rangle_{d\mu_{\alpha}}, 
\qquad f\in L^p(\R_+, \|x\|^{a p} d\mu_{\alpha}).
\end{equation}
\end{theor}
The proofs will be given in the next section, after elaborating necessary tools.
Note that implicitly Theorem \ref{thm:mainunw} asserts the inclusion
$L^p(\R_+,d\mu_{\alpha}) \subset \domain \mathcal{I}^{\alpha,\sigma}$, $1\le p < \infty$.
Similarly, Theorem \ref{thm:main} asserts the inclusion
$L^p(\R_+,\|x\|^{a p}d\mu_{\alpha}) \subset \domain \mathcal{I}^{\alpha,\sigma}$
in all the cases when weighted $L^p-L^q$ boundedness holds. 
Notice also that for $\alpha_o=(-1/2,\ldots,-1/2)$ Theorems \ref{thm:main} and \ref{thm:lag_Hw} coincide,
as it should be, since for $\alpha=\alpha_o$ both settings coincide.
Analogous remark concerns Theorems \ref{thm:mainunw} and \ref{thm:lag_H}.

\begin{corollary} \label{cor:neg3}
Let $\alpha \in [-1\slash 2,\infty)^d$, $\sigma > 0$, $1\le p < \infty$, $1\le q < \infty$, and $(p,q)$
be such a pair that $\eqref{eq32}$ holds.
Then the operator $(\mathcal{L}_\alpha)^{-\sigma}$,
defined on $L^2(\R_+,d\mu_{\alpha})$ by means of the spectral theorem, extends to a bounded operator from 
$L^p(\R_+,d\mu_{\alpha})$ to $L^q(\R_+,d\mu_{\alpha})$. 
Moreover, denoting this extension by $(\mathcal{L}_\alpha)_{pq}^{-\sigma}$, 
for any $k\in\N$ we have
\begin{equation} \label{coeff3}
\langle(\mathcal{L}_\alpha)_{pq}^{-\sigma} f,
\ell^\alpha_k\rangle_{d\mu_{\alpha}}=(\lambda^\alpha_k)^{-\sigma}\langle f,
\ell^\alpha_k\rangle_{d\mu_{\alpha}}, 
\qquad f\in L^p(\R_+,d\mu_{\alpha}).
\end{equation}
\end{corollary}
\begin{proof}
We use Theorem \ref{thm:mainunw} and the arguments from the proof of Corollary \ref{negative}.
We check that $(\mathcal{L}_\alpha)^{-\sigma}= \mathcal{I}^{\alpha,\sigma}$
in $L^2(\R_+,d\mu_{\alpha} )$ by verifying that both operators,
being bounded on $L^2(\R_+,d\mu_{\alpha})$,  coincide  on the dense in $L^2(\R_+,d\mu_{\alpha})$
linear span of Laguerre functions $\ell_k^\alpha$. 
Note that to apply Fubini's theorem it is enough to know that
$\mathcal{K}^{\alpha,\sigma}(x,\cdot)\in L^1(\R_+,d\mu_{\alpha})$ for any fixed $x\in\R_+$, 
and that $\ell^\alpha_k\in L^\infty(\R_+)$. The latter fact is obvious.
To justify the first one observe that due to \eqref{22} and \eqref{dd_H} it suffices to show that
$$
\int_{\R_+}\mathcal{K}^{\sigma}(x,y)\prod_{i=1}^d y_i^{\alpha_i+1/2}\,dy < \infty,
$$
for any fixed $x\in\R$. This, however, easily follows by Proposition \ref{pro:comp}.

Considering \eqref{coeff3}, given $1\le p<\infty$, the subspace 
$L^2(\R_+,d\mu_{\alpha})\cap L^p(\R_+,d\mu_{\alpha})$ is dense in $ L^p(\R_+,d\mu_{\alpha})$, 
hence the extension $(\mathcal{L}_\alpha)_{pq}^{-\sigma}$ coincides with $\mathcal{I}^{\alpha,\sigma}$ 
as a bounded operator from 
$L^p(\R_+,d\mu_{\alpha})$ to $L^q(\R_+,d\mu_{\alpha})$ and therefore 
\eqref{coeff3} follows from \eqref{eq45}.
\end{proof}

Similarly as in the settings of Hermite expansions and Laguerre expansions of Hermite type, 
the following observation is in order. The assertion of Corollary \ref{cor:neg3} remains valid, 
at least in the case when $d=1$ and $\alpha\ge0$, 
and the sequence $(\lambda_k^{\alpha})^{-\sigma}$ in the definition of 
$(\mathcal{L}_\alpha)^{-\sigma}$ is replaced by $(k+1)^{-\sigma}$ 
(then accordingly $\lambda^\alpha_k$ in \eqref{coeff3} must be replaced by $k+1$)
or by any other sequence of similar smoothness. 
This time this is a consequence of a multiplier theorem for one-dimensional
Laguerre expansions of convolution type, see \cite[Theorem 1.1]{STr} (actually that theorem admits 
a weighted setting with power weights involved).

Theorem \ref{thm:main} extends the result of Gasper, Stempak and Trebels \cite[Theorem 1.1]{GST} 
in several directions. To make appropriate comments we first state an equivalent form of 
the theorem from \cite{GST}. In terms of one-dimensional $\{\ell^\alpha_k\}$-expansions it 
reads as follows. Let $\alpha\ge0$,  $1< p\le q<\infty$, $0<\sigma<\alpha+1$,
$a<\frac{2\alpha+2}{p'}$, $b<\frac{2\alpha+2}{q}$, $a+b\ge0$,
$\frac1q=\frac1p-\frac{2\sigma-a-b}{2\alpha+2}$. 
Then the operator $I_\sigma$ defined initially by the series
$$
I_\sigma f= \sum_{k=0}^\infty (k+1)^{-\sigma}\langle f,  \ell^\alpha_k\rangle_{d\mu_{\alpha}} \ell^\alpha_k
$$
on the space spanned by the $\ell^\alpha_k$, $k\ge0$
(hence, in fact, the series terminates), extends to a bounded operator from
$L^p(\mathbb{R}_+,x^{ap}d\mu_{\alpha})$ to $L^q(\mathbb{R}_+,x^{-bq}d\mu_{\alpha})$. 
Thus Theorem \ref{thm:main}, being first of all multi-dimensional, in dimension one releases the
restriction $\sigma<\alpha+1$, enlarges the range of $\alpha$ parameter from $[0,\infty)$ to
$[-1/2,\infty)$, and finally, shows that the constraint $\frac1q=\frac1p-\frac{2\sigma-a-b}{2\alpha+2}$ 
is unnecessary (the possibility of replacing  $\lambda^\alpha_k$ by  $k+1$, at least when $\alpha\ge0$,  
is possible due to the weighted multiplier theorem mentioned above).

\section{Convexity principle} \label{sec:aver}

We start with showing a convexity principle for the heat kernel 
$G_t^\alpha(x,y)$, the potential kernel $\mathcal{K}^{\alpha,\sigma}(x,y)$ and the potential operator
$\mathcal{I}^{\alpha,\sigma}$.

\begin{propo}\label{pro:CP}
Let $\beta,\gamma\in(-1,\infty)^d$, $\beta \neq \gamma$, $\sigma_1,\sigma_2>0$. Further, let
$$
\alpha=\lambda\beta+(1-\lambda)\gamma, \qquad \sigma=\lambda\sigma_1+(1-\lambda)\sigma_2,
$$ 
for some $\lambda\in(0,1)$. Then 
\begin{equation}\label{CW1}
G_t^\alpha(x,y)\simeq \big(G_t^{\beta}(x,y)\big)^\lambda \big(G_t^{\gamma}(x,y)\big)^{1-\lambda},
\qquad x,y\in \R_+,
\end{equation}
\begin{equation}\label{CW2}
\mathcal{K}^{\alpha,\sigma}(x,y) \lesssim \big(\mathcal{K}^{\beta,\sigma_1}(x,y)\big)
^\lambda\big(\mathcal{K}^{\gamma,\sigma_2}(x,y)\big)^{1-\lambda},
\end{equation}
Moreover, for  $f=f_1^\lambda f_2^{1-\lambda}$ with $f_1,f_2\ge0$ we have 
\begin{equation}\label{CW3}
\mathcal{I}^{\alpha,\sigma} f(x) \lesssim
\big(\mathcal{I}^{\beta,\sigma_1}f_1(x)\big)^\lambda\big(\mathcal{I}^{\gamma,\sigma_2}
f_2(x)\big)^{1-\lambda},  \qquad x\in \R_+.
\end{equation}
\end{propo}
\begin{proof}
The first relation is a direct consequence of the explicit formula for $G_t^{\alpha}(x,y)$, 
the asymptotics \eqref{bes} and the continuity of the Bessel functions involved.
The estimate for potential kernels follows from \eqref{CW1} and H\"older's inequality:
\begin{align*}
\mathcal{K}^{\alpha,\sigma}(x,y) & \simeq \int_0^{\infty} 
	\big(G_t^{\beta}(x,y)\big)^\lambda \big(G_t^{\gamma}(x,y)\big)^{1-\lambda}
	t^{\lambda (\sigma_1-1)} t^{(1-\lambda)(\sigma_2-1)} \, dt \\
& \le \bigg( \int_0^{\infty} G_t^{\beta}(x,y) t^{\sigma_1-1}\, dt \bigg)^{\lambda}
	\bigg( \int_0^{\infty} G_t^{\gamma}(x,y) t^{\sigma_2-1}\, dt \bigg)^{1-\lambda} \\
& \simeq \big(\mathcal{K}^{\beta,\sigma_1}(x,y)\big)^\lambda
	\big(\mathcal{K}^{\gamma,\sigma_2}(x,y)\big)^{1-\lambda}.
\end{align*}
Finally, to justify \eqref{CW3} we first observe that 
$$
w_{\alpha}(x)=(w_{\beta}(x))^{\lambda} (w_{\gamma}(x))^{1-\lambda},
$$
where $w_{\alpha}$ denotes the density of the measure $\mu_{\alpha}$.
Then we use \eqref{CW2} and again H\"older's inequality to get
\begin{align*}
\mathcal{I}^{\alpha,\sigma} f(x) & = \int_{\R_+} \mathcal{K}^{\alpha,\sigma}(x,y) f(y) w_{\alpha}(y)\, dy\\
& \lesssim \int_{\R_+} \big(\mathcal{K}^{\beta,\sigma_1}(x,y)\big)^\lambda
	\big(\mathcal{K}^{\gamma,\sigma_2}(x,y)\big)^{1-\lambda} 
	\big(f_1(y)\big)^\lambda \big(f_2(y)\big)^{1-\lambda} 
	(w_{\beta}(y))^{\lambda} (w_{\gamma}(y))^{1-\lambda} \, dy \\
& \le \bigg( \int_{\R_+} \mathcal{K}^{\beta,\sigma_1}(x,y) f_1(y) w_{\beta}(y)\, dy\bigg)^{\lambda}
	\bigg( \int_{\R_+} \mathcal{K}^{\gamma,\sigma_2}(x,y) f_2(y) w_{\gamma}(y)\, dy\bigg)^{1-\lambda} \\
& = \big(\mathcal{I}^{\beta,\sigma_1}f_1(x)\big)^\lambda\big(\mathcal{I}^{\gamma,\sigma_2}
	f_2(x)\big)^{1-\lambda}.
\end{align*}
\end{proof}

We now state the convexity principle that concerns $L^p-L^q$ mapping properties of
the potential operators.

\begin{theor} \label{thm:avg}
Let $\beta,\gamma\in(-1,\infty)^d$, $\beta \neq \gamma$, $\sigma_1,\sigma_2>0$, 
and $1 \le p,q < \infty$. Further, let $U_1,U_2,V_1,V_2$ be strictly positive (up to a set of null measure)
weights on $\R_+$ and put
$$
\alpha=\lambda\beta+(1-\lambda)\gamma, \qquad \sigma=\lambda\sigma_1+(1-\lambda)\sigma_2,
\qquad
U = U_1^{\lambda} U_2^{1-\lambda}, \qquad V = V_1^{\lambda} V_2^{1-\lambda}
$$
for a fixed $\lambda \in (0,1)$. Then boundedness of the operators
$$
\mathcal{I}^{\beta,\sigma_1} \colon L^p(U_1d\mu_{\beta}) \longrightarrow L^{q}(V_1d\mu_{\beta}), \qquad
\mathcal{I}^{\gamma,\sigma_2} \colon L^p(U_2d\mu_{\gamma}) \longrightarrow L^{q}(V_2d\mu_{\gamma}),
$$
implies boundedness of the operator
$$
\mathcal{I}^{\alpha,\sigma} \colon L^p(Ud\mu_{\alpha}) \longrightarrow L^{q}(Vd\mu_{\alpha}).
$$
\end{theor}

\begin{proof}
To obtain the norm estimates it is enough, by the positivity of potential kernels, to consider only
nonnegative functions $f$. By \eqref{CW3}, H\"older's inequality and the assumed boundedness 
of $\mathcal{I}^{\beta,\sigma_1}$ and $\mathcal{I}^{\gamma,\sigma_2}$ it follows that for 
$f = f_1^{\lambda} f_2^{1-\lambda}$
\begin{align*}
\|\mathcal{I}^{\alpha,\sigma}f\|^q_{L^q(Vd\mu_{\alpha})} = 
& \,\int_{\R_+} \big(\mathcal{I}^{\alpha,\sigma}f(x)\big)^q
	V(x) w_{\alpha}(x)\, dx \\
\lesssim  &\, \int_{\R_+} \big(\mathcal{I}^{\beta,\sigma_1} f_1(x)\big)^{\lambda q} 
	\big(V_1(x)w_{\beta}(x)\big)^{\lambda} \big(\mathcal{I}^{\gamma,\sigma_2} f_2(x)\big)^{(1-\lambda) q} 
	\big(V_2(x)w_{\gamma}(x)\big)^{1-\lambda}\, dx \\
\le & \,\bigg( \int_{\R_+} \big(\mathcal{I}^{\beta,\sigma_1} f_1(x)\big)^{q} V_1(x) w_{\beta}(x) \, dx 
	\bigg)^{\lambda} \\
& \,\quad\times \bigg( \int_{\R_+} \big(\mathcal{I}^{\gamma,\sigma_2} f_2(x)\big)^{q} 
	V_2(x) w_{\gamma}(x) \, dx \bigg)^{1-\lambda} \\
\lesssim & \,\bigg( \int_{\R_+} \big(f_1(x)\big)^p U_1(x) w_{\beta}(x) \, dx \bigg)^{\lambda q \slash p}
	\bigg( \int_{\R_+} \big(f_2(x)\big)^p U_2(x) w_{\gamma}(x) \, dx \bigg)^{(1-\lambda) q \slash p}.
\end{align*}
With the choice 
$$
f_1(x) = f(x) \bigg(\frac{U(x) w_{\alpha}(x)}{U_1(x)w_{\beta}(x)}\bigg)^{1\slash p}, \qquad
f_2(x) = f(x) \bigg(\frac{U(x) w_{\alpha}(x)}{U_2(x)w_{\gamma}(x)}\bigg)^{1\slash p},
$$
we get the desired estimate
$$
\|\mathcal{I}^{\alpha,\sigma}f\|_{L^q(Vd\mu_{\alpha})} \lesssim \|f\|_{L^p(Ud\mu_{\alpha})}.
$$
\end{proof}

The importance of the convexity principles comes from the fact that they allow to obtain
weighted mapping properties of the Laguerre potential operators from those for Hermite potential
operators, which are much easier to analyze. More precisely, we shall use a transference
method relating various objects and results between Hermite and Laguerre expansions to obtain
the relevant results in the Laguerre setting, but with the type multi-index restricted to a discrete
set of half-integer multi-indices. Then the convexity principles will enable to interpolate those
``half-integer results'' to cover general $\alpha$.

A detailed description of the transference between the settings of Hermite polynomial expansions and
Laguerre polynomial expansions of half-integer order can be found in \cite{GIT}.
In the present situation, a completely parallel transference method is valid for the Hermite function
setting and the Laguerre function setting of convolution type and half-integer order.
Let $n=(n_1,\ldots,n_d) \in (\mathbb{N}\backslash \{0\})^d$ be a multi-index and
let $y^i = (y_1^i,\ldots,y_{n_i}^i) \in \mathbb{R}^{n_i}$, $i=1,\ldots,d$.
Define the transformation $\phi \colon \mathbb{R}^{|n|} \to \R_{+}$ by
$$
\phi(y^1,\ldots,y^d) = (\|y^1\|,\ldots,\|y^d\|).
$$
Various objects in the Hermite setting in $\mathbb{R}^{|n|}$ and in the 
Laguerre setting in $\R_+$ with the type multi-index $\alpha$ such that 
\begin{equation} \label{half}
\alpha_i=\frac{n_i}{2} -1, \qquad i=1,\ldots,d,
\end{equation}
are connected by means of $\phi$. In particular, we have the following.

\begin{propo} \label{GIT:11}
Let $n$ and $\alpha$ be as above, and $t,\sigma>0$ be fixed. 
Then for any $f$ in the space spanned by the $\ell_k^{\alpha}$ we have
\begin{align*}
(e^{-t\mathcal{L}_{\alpha}}f) \circ \phi(\bar{x}) & = e^{-t\mathcal{H}}(f \circ \phi)(\bar{x}), \qquad 
	\bar{x} \in \mathbb{R}^{|n|},\\
((\mathcal{L}_{\alpha})^{-\sigma}f) \circ \phi(\bar{x}) & 
= \mathcal{H}^{-\sigma}(f \circ \phi)(\bar{x}), \qquad \bar{x} \in \mathbb{R}^{|n|},
\end{align*}
where $\mathcal{L}_{\alpha}$ is the Laguerre operator in $\R_{+}$, and $\mathcal{H}$ 
is the Hermite operator in $\mathbb{R}^{|n|}$.
\end{propo}

\begin{proof}
It suffices to take $f = \ell_k^{\alpha}$ and use the identity
$$
\ell_k^{\alpha} \circ \phi(\bar{x}) = \sum a_r h_{2r}(\bar{x}), \qquad \bar{x} \in \mathbb{R}^{|n|},
$$
where the summation runs over all $r=(r^1,\ldots,r^d)\in \mathbb{N}^{|n|}$ such that 
$|r^i|=k_i$, $i=1,\ldots,d$, see \cite[Lemma 1.1, Proposition 3.1]{GIT}.
\end{proof}

Since $(\mathcal{L}_{\alpha})^{-\sigma} \ell_k^{\alpha} = \mathcal{I}^{\alpha,\sigma} \ell_k^{\alpha}$
and $\mathcal{H}^{-\sigma}h_r = \mathcal{I}^{\sigma}h_r$, see the proofs of Corollaries \ref{cor:neg3}
and \ref{negative}, we get

\begin{corollary} \label{prop:rel}
Assume that $n\in (\mathbb{N}\backslash \{0\})^d$ is related to $\alpha$ as in \eqref{half}.
Given $\sigma>0$, let $\mathcal{I}^{\alpha,\sigma}$ be the Laguerre potential operator in $\R_+$, and
$\mathcal{I}^{\sigma}$ be the Hermite potential operator in $\mathbb{R}^{|n|}$. Then 
$$
(\mathcal{I}^{\alpha,\sigma} f) \circ \phi(\bar{x}) = \mathcal{I}^{\sigma}(f \circ \phi)(\bar{x}),
	\qquad \bar{x} \in \mathbb{R}^{|n|},
$$
for any $f$ in the space spanned by the $\ell_k^{\alpha}$. 
\end{corollary}

The relations of Proposition \ref{GIT:11} and Corollary \ref{prop:rel} hold in fact for more general
functions $f$. This can be seen from connections between heat and potential kernels in both settings.

\begin{propo} \label{heat_rel}
Assume that $n$ and $\alpha$ are related by \eqref{half}. Let 
$G_t^{\alpha}(x,y)$, $\mathcal{K}^{\alpha,\sigma}(x,y)$ be the Laguerre heat and potential kernels in 
$\R_{+}$, and $G_t(\bar{x},\bar{y})$, $\mathcal{K}^{\sigma}(\bar{x},\bar{y})$ the Hermite heat and
potential kernels in $\mathbb{R}^{|n|}$.
Then, for all $\bar{x} \in \mathbb{R}^{|n|}$, $y \in \R_{+}$,
$$
G_t^{\alpha}(\phi(\bar{x}),y) = \int_{S_{n_1-1}} \!\!\!\!\!\!\cdots \int_{S_{n_d-1}} G_t\Big(\bar{x},
	\big(y_1 \xi^1,\ldots,y_d \xi^d\big)\Big) \, d\sigma_1(\xi^1)\ldots d\sigma_d(\xi^d),
$$
hence also
$$
\mathcal{K}^{\alpha,\sigma}
(\phi(\bar{x}),y) = \int_{S_{n_1-1}} \!\!\!\!\!\!\cdots \int_{S_{n_d-1}} \mathcal{K}^{\sigma}\Big(\bar{x},
	\big(y_1 \xi^1,\ldots,y_d \xi^d\big)\Big) \, d\sigma_1(\xi^1)\ldots d\sigma_d(\xi^d),
$$
where $S_{n_i-1}$ is the unit sphere in $\mathbb{R}^{n_i}$ and $\sigma_i$ is the surface measure 
on $S_{n_i-1}$. 
\end{propo}

\begin{proof}
Let $f$ be a linear combination of the $\ell_k^{\alpha}$. Integrating in poly-polar coordinates in 
$\mathbb{R}^{|n|}$ associated to the factorization 
$\mathbb{R}^{|n|}=\mathbb{R}^{n_1}\times \ldots \times \mathbb{R}^{n_d}$
we get, for each $\bar{x} \in \mathbb{R}^{|n|}$,
\begin{align*}
e^{-t\mathcal{H}}(f\circ \phi)(\bar{x}) & 
= \int_{\mathbb{R}^{|n|}} G_t(\bar{x},\bar{y}) f\circ \phi(\bar{y})\, d\bar{y}\\
& = \int_{\R_{+}} \bigg[\int_{S_{n_1-1}} \!\!\!\!\!\!\cdots \int_{S_{n_d-1}} \!\!\!G_t\Big(\bar{x},
	\big(y_1 \xi^1,\ldots,y_d \xi^d\big)\Big) \, d\sigma_1(\xi^1)\ldots d\sigma_d(\xi^d) \bigg]
	f(y) \, d\mu_{\alpha}(y).
\end{align*}
On the other hand, by Proposition \ref{GIT:11} the last expression is equal to
$$
(e^{-t\mathcal{L}_{\alpha}} f) \circ \phi(\bar{x}) = 
\int_{\R_{+}} G_t^{\alpha}(\phi(\bar{x}),y) f(y) \, d\mu_{\alpha}(y).
$$
Thus the desired identity for the heat kernels follows by their continuity and the density in
$L^2(d\mu_{\alpha})$ of the subspace spanned by the $\ell_k^{\alpha}$.
The identity for the potential kernels is an easy consequence of the previous one.
\end{proof}

\begin{corollary} \label{cor:Dom}
Let $\alpha,n,\mathcal{I}^{\sigma},\mathcal{I}^{\alpha,\sigma}$ be as in Corollary \ref{prop:rel}
and let $f$ be a function in $\R_+$.
If $f \circ \phi \in \domain\mathcal{I}^{\sigma}$, then
$f \in \domain\mathcal{I}^{\alpha,\sigma}$ and for almost all $\bar{x} \in \mathbb{R}^{|n|}$
$$
(\mathcal{I}^{\alpha,\sigma}f)\circ \phi(\bar{x}) = \mathcal{I}^{\sigma}(f \circ \phi)(\bar{x}).
$$
\end{corollary}

\begin{proof}
Assuming that $f \circ \phi$ is in the domain of $\mathcal{I}^{\sigma}$ and $\bar{x}$ is not
in an exceptional set of null measure, we iterate the convergent integral
defining $\mathcal{I}^{\sigma}(f \circ \phi)(\bar{x})$ in the poly-polar
coordinates emerging from the factorization 
$\mathbb{R}^{|n|}=\mathbb{R}^{n_1}\times \ldots \times \mathbb{R}^{n_d}$.
Then the pointwise connection from Proposition \ref{heat_rel} between the potential kernels
can be plugged in, and this leads precisely to the desired conclusions.
\end{proof}

An important feature of the transference method is that norm estimates in the two settings in
question are also related. The lemma below makes this precise (compare with \cite[Lemma 2.2]{GIT}).

\begin{lemma} \label{lem:trans}
Let $n$ and $\alpha$ be as in \eqref{half}. 
Assume that $U,V$ are nonnegative weights in $\R_+$, $1\le p,q < \infty$,
and $f$ is a fixed function in $L^p(Ud\mu_{\alpha})$. Suppose that $T,\widetilde{T}$ are operators
defined on $L^p(Ud\mu_{\alpha})$ and $L^p(\mathbb{R}^{|n|},U\circ \phi)$, respectively, satisfying
$(Tf)\circ \phi = \widetilde{T}(f \circ \phi)$. If
$$
\|\widetilde{T}(f\circ\phi)\|_{L^q(\mathbb{R}^{|n|},V\circ\phi)} \le 
	C \|f\circ\phi\|_{L^p(\mathbb{R}^{|n|},U\circ\phi)},
$$
then also
$$
\|{T}f\|_{L^q(Vd\mu_{\alpha})} \le C (\mathcal{C}_{dn})^{1\slash p - 1\slash q}\|f\|_{L^p(Ud\mu_{\alpha})}
$$
with the same constant $C$ and the constant $\mathcal{C}_{dn} = \prod_{i=1}^d \sigma_i(S_{n_i-1})$. 
\end{lemma}

Corollary \ref{cor:Dom} and Lemma \ref{lem:trans} imply the following.

\begin{corollary} \label{cor:trans}
Fix $\sigma>0$. Let $n$ and $\alpha$ satisfy \eqref{half}, and $U,V$ be nonnegative weights in $\R_+$.
If the Hermite potential operator
$$
\mathcal{I}^{\sigma} \colon L^p(\mathbb{R}^{|n|},U\circ \phi) \longrightarrow 
	L^q(\mathbb{R}^{|n|},V\circ \phi)
$$
is bounded, then also bounded is the Laguerre potential operator
$$
\mathcal{I}^{\alpha,\sigma} \colon L^p(Ud\mu_{\alpha}) \longrightarrow L^q(Vd\mu_{\alpha}).
$$
\end{corollary}

Combining this with Theorems \ref{thm:Her} and \ref{thm:Hwt} we get 
\begin{propo} \label{pro:lhalf}
The statements of Theorems \ref{thm:mainunw} and \ref{thm:main} are true for the discrete 
set of half-integer type multi-indices $\alpha$, i.e. for $\alpha$ of the form \eqref{half}.
\end{propo}

Now the convexity principle comes into play, and together with Proposition \ref{pro:lhalf} it allows
to justify Theorems \ref{thm:mainunw} and \ref{thm:main}.
Below we give a proof of Theorem \ref{thm:main} only.
Proving Theorem \ref{thm:mainunw} relies on exactly the same arguments and is in fact simpler
due to the absence of weights.

\begin{proof}[Proof of Theorem \ref{thm:main}]
We shall iterate Theorem \ref{thm:avg} $d$ times, applying it gradually for successive coordinate axes.
Fix $\sigma > 0$. In the first step we deal with the first axis. Let $\alpha$ be of the form
$\alpha=(\alpha_1,\alpha_{2,d})$, where $\alpha_1 \in [-1\slash 2,\infty)$ is arbitrary and 
$\alpha_{2,d} \in [-1\slash 2,\infty)^{d-1}$ is half-integer. In addition, we may assume that $\alpha_1$
is not a half-integer (otherwise Proposition \ref{pro:lhalf} does the job). Then $\alpha$ is
a convex combination of two half-integer multi-indices, say
$$
\alpha = \lambda \beta + (1-\lambda) \gamma, \quad \textrm{where} \quad \beta=(\beta_1,\alpha_{2,d}), 
	\quad \gamma = (\gamma_1,\alpha_{2,d}),
$$
for some half-integers $\beta_1,\gamma_1 \ge -1\slash 2$ and certain $\lambda \in (0,1)$. Take
$$
\sigma_\beta = \sigma \frac{|\beta|+d}{|\alpha|+d}, \qquad 
\sigma_\gamma = \sigma \frac{|\gamma|+d}{|\alpha|+d}
$$
satisfying $\sigma=\lambda \sigma_\beta + (1-\lambda)\sigma_\gamma$.
Further, given $a$ and $b$, take 
$$
{a_{\beta}}=a\frac{|\beta|+d}{|\alpha|+d},\qquad 
{b_{\beta}}=b\frac{|\beta|+d}{|\alpha|+d}, \qquad
{a_{\gamma}}=a\frac{|\gamma|+d}{|\alpha|+d}, \qquad
{b_{\gamma}}=b\frac{|\gamma|+d}{|\alpha|+d},
$$
and notice that $a=\lambda a_{\beta}+(1-\lambda) a_{\gamma}$ and $b=\lambda
b_{\beta}+(1-\lambda)b_{\gamma}$.
Now the crucial observation is that the assumptions of Theorem \ref{thm:main} are satisfied 
with the parameters $(d,p,q,\alpha,a,b)$ if and only if they are satisfied with the parameters
$(d,p,q,\beta,a_{\beta},b_{\beta})$ and simultaneously with the parameters
$(d,p,q,\gamma,a_{\gamma},b_{\gamma})$.
Moreover, the splitting into cases (i) and (ii) in Theorem \ref{thm:main} with parameters
$(d,p,q,\alpha,a,b)$ coincides with the two splittings determined by the parameters
$(d,p,q,\beta,a_{\beta},b_{\beta})$ and $(d,p,q,\gamma,a_{\gamma},b_{\gamma})$.
Thus, in virtue of Theorem \ref{thm:avg} and Proposition \ref{pro:lhalf}, we infer the desired boundedness
of $\mathcal{I}^{\alpha,\sigma}$.

In the second step we let $\alpha$ be of the form $\alpha=(\alpha_1,\alpha_2,\alpha_{3,d})$, where
$\alpha_1,\alpha_2 \in[-1\slash 2,\infty)$ are arbitrary and $\alpha_{3,d}\in [-1\slash 2,\infty)^{d-2}$ is
half-integer
(we may assume that $\alpha_2$ is not a half-integer). Then $\alpha$ is a convex combination,
$$
\alpha = \lambda \beta + (1-\lambda) \gamma, \quad \textrm{with} \quad 
\beta=(\alpha_1,\beta_2,\alpha_{3,d}), \quad \gamma = (\alpha_1,\gamma_2,\alpha_{3,d}),
$$ 
for some half-integers $\beta_2,\gamma_2 \ge -1\slash 2$ and certain $\lambda \in (0,1)$.
Choosing $\sigma_{\beta},\sigma_{\gamma},a_{\beta},b_{\beta},a_{\gamma},b_{\gamma}$ by means of the
formulas from the previous step, and using the already proved partial result, we get suitable boundedness
properties of $\mathcal{I}^{\sigma_{\beta},\beta}$ and $\mathcal{I}^{\sigma_{\gamma},\gamma}$.
Then Theorem \ref{thm:avg} gives the required boundedness properties of $\mathcal{I}^{\alpha,\sigma}$.

In the next steps we repeat the procedure until the last coordinate of $\alpha$ is reached, in each step
using a partial result obtained in the preceding step. After $d$ steps the theorem is proved.

Finally, proving \eqref{co44} we only need to know that 
$\|\ell_k^\alpha\|_{L^{q'}(\|x\|^{bq'}d\mu_\alpha)}<\infty$, but this holds in view of the inequality
$b>-2(|\alpha|+d)/q'$ following from the assumptions imposed on $d,p,q,\alpha,a,b$.
\end{proof}

We finish this section by noting that weak type estimates in the setting of Hermite function
expansions and of Laguerre expansions of convolution type are also related.
In fact statements analogous to Lemma \ref{lem:trans} and Corollary \ref{cor:trans} are true for
weak type inequalities. This allows to transfer the weak type $(1,q)$ result from Theorem \ref{thm:Her}
to the Laguerre settings of half-integer orders $\alpha$.
Unfortunately, there seems to be no tool similar to the convexity principle that would enable
to interpolate the weak type estimates to all intermediate $\alpha$.

\section{Negative powers of the Dunkl  harmonic oscillator} \label{sec:dunkl}

In this section we take an opportunity to continue the study of spectral properties of the Dunkl 
harmonic oscillator in the context of a finite reflection group acting on $\R$ and
isomorphic to $\mathbb{Z}^d_2=\{0,1\}^d$. This completes in some sense our investigation began in
\cite{NS3} and continued in \cite{NS4}. We now recall the most relevant ingredients of the setting, 
kindly referring the reader to 
\cite{NS4} for a more detailed description. We keep the notation used in \cite{NS3,NS4}, but to avoid a
possible collision with notation of the present paper, in several places we use the letter $D$ to
indicate the Dunkl setting.

The Dunkl harmonic oscillator
\begin{equation*}
L_{\alpha}^D=-\Delta_{\alpha}+\|x\|^2
\end{equation*}
is a differential-difference operator in $\R$, 
where $\Delta_{\alpha}$ denotes the Dunkl Laplacian in the context of a finite  reflection group on $\R$ 
isomorphic to $\mathbb{Z}^d_2$. Here $\alpha\in[-1/2,\infty)^d$ plays a role of the multiplicity function.
$L_{\alpha}^D$
considered initially on the Schwartz class $\mathcal{S}(\R)$ as the natural domain is 
symmetric and positive in $L^2(\R,w_\alpha)$, where
\begin{equation*}
w_\alpha(x)= \prod_{j=1}^d|x_j|^{2\alpha_j+1}, \qquad x\in \R
\end{equation*}
(notice that $w_{\alpha}$ restricted to $\R_+$ is precisely the density of the measure $\mu_{\alpha}$
related to the Laguerre setting from Sections \ref{sec:lag} and \ref{sec:aver}).
The associated generalized Hermite functions are tensor products
$$
h_{k}^{\alpha}(x) = h _{k_1}^{\alpha _1}(x_1) \cdot \ldots \cdot
h_{k_d}^{\alpha _d}(x_d), \qquad x \in \R, \quad k \in \mathbb{N}^d,
$$
where $h_{k_i}^{\alpha _i}$ are the one-dimensional functions
\begin{equation*}
h_{2k_i}^{\alpha_i}(x_i)=d_{2k_i,\alpha_i}e^{-x_i^2/2}L_{k_i}^{\alpha_i}(x_i^2), \qquad
h_{2k_i+1}^{\alpha _i}(x_i)=d_{2k_i+1,\alpha_i}e^{-x_i^2/2}x_iL_{k_i}^{\alpha_i+1}(x_i^2),
\end{equation*}
with
$$ 
d_{2k_i,\alpha_i}=(-1)^{k_i}\bigg(\frac{\Gamma(k_i+1)}{\Gamma(k_i+\alpha_i +1)}\bigg)^{1/2}, \qquad
d_{2k_i+1,\alpha_i}=(-1)^{k_i}\bigg(\frac{\Gamma(k_i+1)}{\Gamma(k_i+\alpha_i +2)}\bigg)^{1/2}.
$$ 
The system $\{h_k^{\alpha}: k\in\N\}$ is orthonormal and complete in $L^2(\mathbb{R}^d, w_{\alpha})$,
and ${L}_{\alpha}^D h_k^{\alpha}=(2|k|+2|\alpha|+2d)h_k^{\alpha}$.

Let $\mathcal{L}_{\alpha}^D$ be the self-adjoint extension of $L_{\alpha}^D$, as defined in 
\cite[p.\,542]{NS3} or \cite[p.\,3]{NS4}.
The negative power  $(\mathcal{L}_{\alpha}^D)^{-\sigma}$, $\sigma>0$, is given on  $L^2(\R,w_\alpha)$
by the spectral series
\begin{equation*}
  \big(\mathcal{L}_{\alpha}^D\big)^{-\sigma} f=\sum_{k\in \N} (2|k|+2|\alpha|+2d)^{-\sigma}\langle
   f,h_k^{\alpha}
  \rangle_{dw_\alpha} \, h_k^{\alpha};
\end{equation*}
here 
$\langle\cdot,\cdot
  \rangle_{dw_\alpha}$ denotes the canonical inner product in $L^2(\mathbb{R}^d, w_{\alpha})$. 
Clearly, $(\mathcal{L}_{\alpha}^D)^{-\sigma}$ is a contraction in $L^2(\R,w_{\alpha})$.
For $\alpha_o=(-1/2,\ldots,-1/2)$ one recovers the setting of the classic harmonic oscillator:
$\Delta_{\alpha_o}$ becomes the Euclidean Laplacian, $w_{\alpha_o}\equiv1$ and $h_k^{\alpha_o}$
are the usual Hermite functions.
  
The semigroup $\{e^{-t\mathcal{L}_{\alpha}^D} : t>0\}$ generated by
$\mathcal{L}_{\alpha}^D$  has the integral representation
\begin{equation*}
  e^{-t\mathcal{L}_{\alpha}^D} f(x)=\int_{\R} G_t^{\alpha,D}(x,y)f(y)\,dw_{\alpha}(y),
  \qquad x\in \R,
\end{equation*}
where the heat kernel is given by
\begin{equation*}
G^{\alpha,D}_t(x,y)=\sum_{n=0}^\infty e^{-t(2n+2|\alpha|+2d)} \sum_{|k|=n}
h_k^\alpha(x)h_k^\alpha(y).
\end{equation*}
The oscillating series defining $G_t^{\alpha,D}(x,y)$ can be summed, and the resulting expression
involves the modified Bessel function $I_{\nu}$, see \cite{NS3} or \cite{NS4}.
Here, however, we shall not need explicitly that formula.
To obtain $L^p-L^q$ estimates for $(\mathcal{L}_{\alpha}^D)^{-\sigma}$ we will make use of
the estimate, cf. \cite[p.\,545]{NS3},
\begin{equation} \label{compld}
G_t^{\alpha,D}(x,y) \lesssim G_t^{\alpha}(|x|,|y|), \qquad x,y \in \R, \quad t>0,
\end{equation}
where $|x|=(|x_1|,\ldots,|x_d|)$, for $x=(x_1,\ldots,x_d)\in\R$, and the kernel
$G_t^{\alpha}(x,y)$ is understood to be given by, see \eqref{11},
$$
G^{\alpha}_t(x,y)=(\sinh 2t)^{-d-|\alpha|}
\exp\Big({-\frac{1}{2} \coth(2t)\big(\|x\|^{2}+\|y\|^{2}\big)}\Big)
\prod^{d}_{i=1} \Big(\frac{x_i y_i}{\sinh 2t}\Big)^{-\alpha_i} 
I_{\alpha_i}\left(\frac{x_i y_i}{\sinh 2t}\right).
$$
The following comment is in order at this point.
The modified Bessel function of the first kind and order $\nu$ is defined by
\begin{equation*}
I_{\nu}(z) = \sum_{k=0}^{\infty} \frac{(z\slash 2)^{\nu+2k}}{\Gamma(k+1)\Gamma(k+\nu+1)}.
\end{equation*}
Here we consider the function $z \mapsto z^{\nu}$, and thus also the Bessel function $I_{\nu}(z)$,
as an analytic function defined on $\C \backslash \{ix : x \le 0\}$ (usually $I_{\nu}$ is considered
as a function on $\C$ cut along the half-line $(-\infty,0]$). Hence $z^{-\nu}I_\nu(z)$ is readily
extended to an entire function and $s^{-\nu}I_{\nu}(s)$, as a function on $\mathbb{R}$,
is real even positive and smooth for any $\nu > -1$, see \cite[Chapter 5]{Leb}.
In particular we see that the function $G_t^{\alpha}(|x|,|y|)$ is well defined 
by the above formula also at
points where $x_i=0$ or $y_i=0$ for some $i=1,\ldots,d$.
For future reference it is convenient to call such points \emph{critical}
($x\in \R$ is critical if and only if $x_i=0$ for some $i=1,\ldots,d$).
Notice that critical points in $\R$ form a set of null measure. 

We now define the potential kernel
\begin{equation*}
\mathcal{K}^{\alpha,\sigma}_D(x,y) = \frac1{\Gamma(\sigma)}\int_0^\infty G^{\alpha,D}_t(x,y)
 t^{\sigma-1}\,dt, \qquad x,y \in \R,
\end{equation*}
and the potential operator
$$
\mathcal{I}^{\alpha,\sigma}_D f(x) = \int_{\R}\mathcal{K}^{\alpha,\sigma}_D(x,y)f(y)\,dw_\alpha(y)
$$
with the natural domain $\domain \mathcal{I}^{\alpha,\sigma}_D$.
Since by \eqref{compld}
\begin{equation} \label{compot}
\mathcal{K}^{\alpha,\sigma}_D (x,y) \lesssim \mathcal{K}^{\alpha,\sigma}(|x|,|y|), \qquad x,y \in \R,
\end{equation}
(for the critical points $\mathcal{K}^{\alpha,\sigma}(|x|,|y|)$ is defined by integrating
$G_t^{\alpha}(|x|,|y|)$) it follows from remarks concerning well-definess of the Laguerre potential kernel
$\mathcal{K}^{\alpha,\sigma}(x,y)$ that 
$\mathcal{K}^{\alpha,\sigma}_D(x,y)$ is well defined (the relevant integral is convergent)
whenever $|x|\neq|y|$ and $x,y$ are not critical.
The well-definess at the critical points, still for $|x|\neq|y|$, 
can be easily seen by analysis of the kernel $G_t^{\alpha}(|x|,|y|)$,
taking into account the product structure of the kernel, the above comment concerning non-critical points
and the fact that the function $s \mapsto s^{-\nu}I_{\nu}(s)$, $\nu > -1$, has a finite
value at $s=0$.
We remark that a more detailed analysis reveals that $\mathcal{K}^{\alpha,\sigma}_D (x,y)$ 
is well defined for all $x,y \in \R$ provided that $\sigma > d + |\alpha|$.

As consequences of Theorems \ref{thm:mainunw} and \ref{thm:main} 
we get the following results in the Dunkl setting.
\begin{theor} \label{thm:dunklunw}
Assume that $\alpha \in [-1\slash 2,\infty)^d$. Let $\sigma > 0$ and $1\le p < \infty$, $1\le q < \infty$.
If $\sigma \ge |\alpha|+d$, then
\begin{equation} \label{eq32d}
\|\mathcal{I}_D^{\alpha,\sigma}f\|_{L^q(dw_{\alpha})} \lesssim \|f\|_{L^p(dw_{\alpha})}, \qquad
f \in L^p(\R,dw_{\alpha}).
\end{equation}
If $0 < \sigma < |\alpha|+d$, then \eqref{eq32d} holds under the additional condition
$$
\frac{1}{p}-\frac{\sigma}{|\alpha|+d} \le \frac{1}{q} < \frac{1}{p} + \frac{\sigma}{|\alpha|+d},
$$
with exclusion of the case when $p=1$ and $q=\frac{|\alpha|+d}{|\alpha|+d-\sigma}$.
Moreover, under the assumptions ensuring \eqref{eq32d},
\begin{equation*} 
\langle\mathcal{I}_D^{\alpha,\sigma} f, h^\alpha_k\rangle_{dw_{\alpha}}
=(2|k|+2|\alpha|+2d)^{-\sigma}\langle f, h^\alpha_k\rangle_{dw_{\alpha}}, 
\qquad f\in L^p(\R, dw_{\alpha}).
\end{equation*}
\end{theor}
\begin{theor} \label{thm:dunkl}
Assume that $\alpha \in [-1\slash 2, \infty)^d$. 
Let $\sigma > 0$, $1<p \le q < \infty$ and 
\begin{equation*}
a < (2|\alpha|+2d)\slash p', \qquad 
b < (2|\alpha|+2d)\slash q, \qquad a+b \ge 0.
\end{equation*}
\begin{itemize}
\item[(i)]
If $\sigma \ge |\alpha|+d$, then $\mathcal{I}_D^{\alpha,\sigma}$ maps boundedly 
$L^p(\R,\|x\|^{a p}dw_{\alpha})$ into $L^q(\R,\|x\|^{-b q}dw_{\alpha})$. 
\item[(ii)]
If $\sigma < |\alpha|+d$, then the same boundedness holds under the additional condition
\begin{equation*} 
\frac{1}{q} \ge \frac{1}{p} - \frac{2\sigma-a-b}{2|\alpha|+2d}.
\end{equation*}
\end{itemize}
Moreover, under the assumptions ensuring boundedness of $\mathcal{I}_D^{\alpha,\sigma}$ 
from $L^p(\R,\|x\|^{a p}dw_{\alpha})$ into $L^q(\R,\|x\|^{-b q}dw_{\alpha})$, 
\begin{equation} \label{co44du}
\langle\mathcal{I}_D^{\alpha,\sigma} f, h^\alpha_k\rangle_{dw_{\alpha}}
=(2|k|+2|\alpha|+2d)^{-\sigma}\langle f, h^\alpha_k\rangle_{dw_{\alpha}}, 
\qquad f\in L^p(\R, \|x\|^{a p}dw_{\alpha}).
\end{equation}
\end{theor}

Note that implicitly Theorem \ref{thm:dunklunw} asserts the inclusion
$L^p(\R,dw_{\alpha}) \subset \domain \mathcal{I}_D^{\alpha,\sigma}$, $1\le p < \infty$.
Similarly, Theorem \ref{thm:dunkl} asserts the inclusion
$L^p(\R,\|x\|^{a p}dw_{\alpha}) \subset \domain \mathcal{I}_D^{\alpha,\sigma}$
in all the cases when weighted $L^p-L^q$ boundedness holds. 
Notice also that for $\alpha_o=(-1/2,\ldots,-1/2)$ Theorems \ref{thm:dunkl} and \ref{thm:Hwt} coincide,
as it should be, since for $\alpha=\alpha_o$ both settings coincide. 
Thus the present result generalizes Theorem \ref{thm:Hwt}.
Moreover, Theorem \ref{thm:dunkl} can be regarded as an extension of Theorem \ref{thm:main}
since the Laguerre potential operator $\mathcal{I}^{\alpha,\sigma}$ applied to a function $f$ on $\R_+$
corresponds to $\mathcal{I}_D^{\alpha,\sigma}$ applied to the $\mathbb{Z}_2^d$-symmetric 
(reflection invariant) extension of $f$ to $\R$.
Analogous comments concern relations between Theorem \ref{thm:dunklunw} and Theorems \ref{thm:Her}
and \ref{thm:mainunw}.

We now give a proof of Theorem \ref{thm:dunkl}.
Proving Theorem \ref{thm:dunklunw} relies on exactly the same arguments and Theorem \ref{thm:mainunw},
hence we omit the details.

\begin{proof}[Proof of Theorem \ref{thm:dunkl}]
It is sufficient to consider only nonnegative functions on $\R$; let $f$ be such a function.
Given $\varepsilon \in \mathbb{Z}_2^d$, denote by $f_{\varepsilon}$ the $\varepsilon$-reflection
of $f$, that is
$$
f_{\varepsilon}(y) = f\big( (-1)^{\varepsilon_1}y_1,\ldots,(-1)^{\varepsilon_d}y_d \big), \qquad y \in \R.
$$
Noticing that the density $w_{\alpha}$ is reflection-invariant and using \eqref{compot} we get
\begin{align*}
\mathcal{I}^{\alpha,\sigma}_D f(x) & \lesssim \int_{\R} \mathcal{K}^{\alpha,\sigma}(|x|,|y|)f(y)\,
	dw_{\alpha}(y) \\
& = \sum_{\varepsilon \in \mathbb{Z}_2^d} \int_{\R_+} \mathcal{K}^{\alpha,\sigma}(|x|,y)
		f_{\varepsilon}(y) \, d\mu_{\alpha}(y) \\
& = \sum_{\varepsilon \in \mathbb{Z}_2^d} \mathcal{I}^{\alpha,\sigma} f_{\varepsilon}(|x|),
\end{align*}
where $x\in \R$ is assumed to be non-critical and $f_{\varepsilon}$ in the last occurrence is
understood as the restriction of $f_{\varepsilon}$ to $\R_+$.
Since power weights in $\R$, as well as the function 
$\mathcal{I}^{\alpha,\sigma} f_{\varepsilon}(|\cdot|)$, are reflection-invariant, it follows that
$$
\|\mathcal{I}^{\alpha,\sigma}_D f \|_{L^q(\R,\|x\|^{-bq}dw_{\alpha})} \lesssim 
	\sum_{\varepsilon \in \mathbb{Z}_2^d} 
	\|\mathcal{I}^{\alpha,\sigma} f_{\varepsilon}\|_{L^q(\R_+,\|x\|^{-bq}d\mu_{\alpha})}.
$$ 
Observing now that
$$
\sum_{\varepsilon \in \mathbb{Z}_2^d}
	\|f_{\varepsilon}\|_{L^p(\R_+,\|x\|^{ap}d\mu_{\alpha})} \simeq \|f\|_{L^p(\R,\|x\|^{ap}dw_{\alpha})},
$$
we see that the first part of Theorem \ref{thm:dunkl} is readily deduced from Theorem \ref{thm:main}.

To prove \eqref{co44du} we argue as in the other settings.
Then we only need to know that $\|h_k^{\alpha}\|_{L^{q'}(\|x\|^{bq'}dw_{\alpha})} < \infty$,
but this follows from the explicit form of the functions $h_k^{\alpha}$
and the assumptions imposed on $d,p,q,\alpha,a,b$.
\end{proof}

As a consequence of Theorem \ref{thm:dunklunw} we state the following.
\begin{corollary}
Let  $\alpha \in [-1/2,\infty)^{d}$, $\sigma>0$, $1\le p < \infty$, $1\le q < \infty$, and $(p,q)$ be
such a pair that \eqref{eq32d} holds.
Then $(\mathcal{L}^D_\alpha)^{-\sigma}$ 
extends to a bounded operator from $L^p(\R,dw_{\alpha})$ to $L^q(\R,dw_{\alpha})$. 
Moreover, denoting this extension by $(\mathcal{L}^D_\alpha)_{pq}^{-\sigma}$, 
for any $k\in\N$ we have
\begin{equation*}
\langle(\mathcal{L}^D_\alpha)_{pq}^{-\sigma} f,
h^\alpha_k\rangle_{dw_{\alpha}}=(2|k|+2|\alpha|+2d)^{-\sigma}\langle f,
h^\alpha_k\rangle_{dw_{\alpha}}, 
\qquad f\in L^p(\R,dw_{\alpha}).
\end{equation*}
\end{corollary}

\begin{proof}
We repeat the reasoning from the proof of Corollary \ref{negative}.
The only point that requires further comments is the application of Fubini's theorem.
To make it possible one has to verify that for each $x \in \R$ and each $k \in \mathbb{N}^d$
$$
\int_{\R} \mathcal{K}_D^{\alpha,\sigma}(x,y) |h_k^{\alpha}(y)| \, dw_{\alpha}(y) < \infty.
$$
When $x$ is not a critical point, this holds because $h_k^{\alpha} \in L^{\infty}(\R)$ and
$\mathcal{K}_D^{\alpha,\sigma}(x,\cdot) \in L^1(\R,dw_{\alpha})$ (the first fact is obvious
by the definition of $h_k^{\alpha}$, and the second one follows taking into account 
\eqref{compot} and that
$\mathcal{K}^{\alpha,\sigma}(|x|,\cdot) \in L^1(\R_+,d\mu_{\alpha})$, as verified in the proof of
Corollary \ref{cor:neg3}).

Assume now that $x\neq (0,\ldots,0)$ is a critical point.
By symmetry reasons we may only consider $x$ of the form
$$
x = (0,\ldots,0,x_{n+1},\ldots,x_d), \qquad x_{n+1}\neq 0,\ldots,x_d \neq 0,
$$
for $n=1,\ldots,d-1$.
We will show that $\mathcal{K}_D^{\alpha,\sigma}(x,\cdot) \in L^1(\R,e^{-\|y\|^2\slash 4}dw_{\alpha})$;
this will be enough since $e^{\|\cdot\|^2\slash 4}h_k^{\alpha} \in L^{\infty}(\R)$.
To do that, we first estimate the heat kernel.
By \eqref{compld} and the product structure of $G_t^{\alpha}(x,y)$, for non-critical $y$ we have
$$
G_t^{\alpha,D}(x,y) \lesssim \bigg( \prod_{i=1}^n (\sinh 2t)^{-\alpha_i-1} 	
\exp\Big( -\frac{1}{2} \coth(2t) y_i^2 \Big) \bigg) 
G_t^{\widetilde{\alpha}}(|\widetilde{x}|,|\widetilde{y}|),
$$
where $G_t^{\widetilde{\alpha}}$ is the Laguerre heat kernel in $\mathbb{R}^{d-n}$,
$\widetilde{\alpha}=(\alpha_{n+1},\ldots,\alpha_d)$, $\widetilde{x}=(x_{n+1},\ldots,x_d)$
and $\widetilde{y}=(y_{n+1},\ldots,y_d)$ (notice that by assumption $\widetilde{x}$ is not a critical
point in $\mathbb{R}^{d-n}$). Consequently,
\begin{align*}
\mathcal{K}_D^{\alpha,\sigma}(x,y) \lesssim & \; \int_0^1 \bigg( \prod_{i=1}^n t^{-\alpha_i-1}
	e^{-y_i^2\slash (4t)} t^{\sigma \slash (2n)} \bigg) 
	G_t^{\widetilde{\alpha}}(|\widetilde{x}|,|\widetilde{y}|) t^{\sigma\slash 2 -1} dt \\
	& + \; \int_1^{\infty} \bigg( \prod_{i=1}^n e^{-2t(\alpha_i+1)} e^{-y_i^2\slash 2} \bigg)
	G_t^{\widetilde{\alpha}}(|\widetilde{x}|,|\widetilde{y}|) t^{\sigma -1} dt \\
\lesssim &\; \bigg(\prod_{i=1}^n |y_i|^{-(2\alpha_i+2)+\sigma\slash n}\bigg) \int_0^1
	G_t^{\widetilde{\alpha}}(|\widetilde{x}|,|\widetilde{y}|) t^{\sigma\slash 2 -1} dt \\
	& + \; \bigg( \prod_{i=1}^n e^{-y_i^2\slash 2}\bigg) \int_1^{\infty}
	G_t^{\widetilde{\alpha}}(|\widetilde{x}|,|\widetilde{y}|) t^{\sigma\slash 2 -1} dt \\
\le &\; g(y_1,\ldots,y_n) \,
\mathcal{K}^{\widetilde{\alpha},\sigma\slash 2}(|\widetilde{x}|,|\widetilde{y}|),
\end{align*}
where $g(y_1,\ldots,y_n) = \prod_{i=1}^n |y_i|^{-(2\alpha_i+2)+\sigma\slash n}
	+ \prod_{i=1}^n e^{-y_i^2\slash 2}$. Since $g$ is integrable against the measure
$\prod_{i=1}^n |y_i|^{2\alpha_i+1} e^{-y_i^2\slash 4}dy_i$ and
$\mathcal{K}^{\widetilde{\alpha},\sigma\slash 2}(|\widetilde{x}|,\cdot) \in L^1(\mathbb{R}^{d-n},
d\mu_{\widetilde{\alpha}})$, the conclusion follows.

Finally, for $x=(0,\ldots,0)$ and non-critical $y$ we have
$$
\mathcal{K}_D^{\alpha,\sigma}(x,y) \lesssim g(y_1,\ldots,y_d),
$$
hence again $\mathcal{K}_D^{\alpha,\sigma}(x,\cdot) \in L^1(\R,e^{-\|y\|^2\slash 4}dw_{\alpha})$.
\end{proof}

\section{Final observations and remarks} \label{sec:final}

It is interesting to observe that in dimension one the weighted results from Theorem \ref{thm:main}
imply new weighted results for the potentials $\mathcal{I}_H^{\alpha,\sigma}$, which are different from
those in Theorem \ref{thm:lag_Hw}. Here are the details. 
Recall that the operator $S_{\alpha}$ given by $S_{\alpha}f(x)=x^{\alpha+1\slash 2}f(x)$
intertwines the potential operators in both settings, 
$$
\mathcal{I}_H^{\alpha,\sigma} \circ S_{\alpha} = S_{\alpha} \circ \mathcal{I}^{\alpha,\sigma},
$$
and the $L^p(Ud\mu_\alpha)\to L^q(Vd\mu_\alpha)$ estimate for $\mathcal{I}^{\alpha,\sigma}$ is equivalent to
the $L^p(\widetilde{U})\to L^q(\widetilde{V})$ estimate for $\mathcal{I}_H^{\alpha,\sigma}$, where
$$
\widetilde{U}(x) = U(x) x^{-(\alpha+1\slash 2)p + 2\alpha+1}, \qquad 
\widetilde{V}(x) = V(x) x^{-(\alpha+1\slash 2)q + 2\alpha+1}.
$$
This allows to translate directly the results of Theorem \ref{thm:main} (specified to $d=1$)
to the setting of Laguerre expansions of Hermite type. Thus we get the following.

\begin{propo} \label{prop:ell}
Assume that $d=1$ and $\alpha \ge -1\slash 2$. 
Let $\sigma > 0$, $1<p \le q < \infty$ and 
$$
A < \frac{1}{p'}+\alpha+\frac{1}{2}, \qquad 
B < \frac{1}{q} + \alpha+\frac{1}{2}, \qquad A+B \ge (2\alpha+1)\bigg(\frac{1}{p}-\frac{1}{q}\bigg).
$$
\begin{itemize}
\item[(i)]
If $\sigma \ge \alpha+1$, then $\mathcal{I}^{\alpha,\sigma}_H$ maps boundedly 
$L^p(\mathbb{R}_+,x^{A p})$ into $L^q(\mathbb{R}_+,x^{-B q})$. 
\item[(ii)]
If $\sigma < \alpha+1$, then the same boundedness is true under the additional condition
\begin{equation*}
\frac{1}{q} \ge \frac{1}{p} + A+B - 2\sigma.
\end{equation*}
\end{itemize}
\end{propo}

Notice that here the unweighted case $A=B=0$ is not admitted when $\alpha > -1\slash 2$ and $p<q$,
whereas it is covered by Theorem \ref{thm:lag_Hw}. Also, the splittings into cases (i) and (ii) 
in Proposition \ref{prop:ell} and Theorem \ref{thm:lag_Hw} are different.
In fact, by the abovementioned equivalence, Theorems \ref{thm:lag_Hw} and \ref{thm:main}
are independent in the sense that neither of them follows from the other one.
In particular, at least in dimension one, Theorem \ref{thm:lag_Hw} implies some weighted $L^p-L^q$ estimates
for $\mathcal{I}^{\alpha,\sigma}$ which are not covered by Theorem \ref{thm:main}.
This indicates that Theorem \ref{thm:main} is not optimal in the sense of admissible
power weights. 

In \cite{KS} Kanjin and E. Sato proved a fractional integration theorem for one-dimensional
expansions with respect to the so-called standard system of Laguerre functions
$\{\mathcal{L}_k^{\alpha}\}_{k\ge0}$, where
$$
\mathcal{L}_k^{\alpha}(x)=\left(\frac{\Gamma (k+1)}{\Gamma(k+\alpha +1)}\right)
^{1 \slash 2} L_{k}^{\alpha} (x) x^{\alpha\slash 2}e^{-{x}/{2}}.
$$
This system is orthonormal and complete in $L^2(\mathbb R_+)$. 
The fractional integral operator $I^{\alpha}_{\sigma}$ considered
in \cite{KS} is defined by
\begin{equation}\label{Kan}
I^\alpha_\sigma f\sim \sum_{k=1}^\infty \frac 1{k^\sigma} a_k^\alpha(f)\mathcal{L}_k^{\alpha},
\end{equation}
where $a_k^\alpha(f)=\int_0^\infty f(x)\mathcal{L}_k^{\alpha}(x)\,dx$
are the Fourier-Laguerre coefficients of $f$ (provided they exist), and the sign $\sim$ means that
the coefficients of a function on the left-hand side of \eqref{Kan}
coincide with those appearing on the right-hand side.

The main result of \cite{KS} says that, given $0<\sigma<1$ and $\alpha>-1$, the inequality
\begin{equation}\label{Kan2}
\|I^\alpha_\sigma f\|_q \lesssim \|f\|_p
\end{equation}
holds provided that $1<p<q<\infty$ and $\frac1q=\frac1p-\sigma$, with the additional restriction
$(1+\frac\alpha2)^{-1}<p<q<-2/\alpha$ in the case when $-1<\alpha<0$ (this restriction in particular 
guarantees existence of the coefficients $a_k^\alpha(f)$ and $a_k^\alpha(g)$ for every 
$f\in L^p(\mathbb R_+)$ and every $g\in L^q(\mathbb R_+)$).
We shall briefly see that for $\alpha \ge -1\slash 2$ our present results extend in several directions
the result of Kanjin and Sato. To this end we always assume that $d=1$. 

Each $\mathcal{L}_k^{\alpha}$ is an eigenfunction of the differential operator
$$
\mathbb{L}_{\alpha} = - x \frac{d^2}{dx^2} - \frac{d}{dx} + \frac{x}{4} + \frac{\alpha^2}{4x},
$$
which is formally symmetric and positive in $L^2(\mathbb{R}_+)$, 
and the corresponding eigenvalue is $k+\alpha\slash 2 + 1\slash 2$. 
Moreover, $\mathbb{L}_{\alpha}$ has the natural self-adjoint extension
in $L^2(\mathbb{R}_{+})$ for which the spectral decomposition is given by the $\mathcal{L}_k^{\alpha}$.
We consider the potential operator 
$$
\mathcal{I}_S^{\alpha,\sigma} f(x) = \int_0^{\infty} \mathcal{K}_S^{\alpha,\sigma}(x,y) f(y) \, dy,
$$
where, as in the other settings, the potential kernel is obtained by integrating the associated heat
kernel,
$$
\mathcal{K}_S^{\alpha,\sigma}(x,y) = \frac{1}{\Gamma(\sigma)} \int_0^{\infty} G_t^{\alpha,S}(x,y)
	t^{\sigma-1} \, dt.
$$
Since $\varphi_k^{\alpha}(x) = \sqrt{2x}\mathcal{L}_k^{\alpha}(x^2)$ and the eigenvalues in both
settings coincide up to the constant factor $4$, it follows that 
$G_t^{\alpha,H}(x,y) = 2\sqrt{xy} G_{4t}^{\alpha,S}(x^2,y^2)$ and, consequently,
$\mathcal{K}_H^{\alpha,\sigma}(x,y) = 2^{-2\sigma+1} \sqrt{xy} \mathcal{K}_S^{\alpha,\sigma}(x^2,y^2)$.
This shows that the potential operators in the settings of standard Laguerre expansions and Laguerre
expansions of Hermite type are related, namely
$$
P \circ \mathcal{I}_S^{\alpha,\sigma} = 2^{2\sigma} \mathcal{I}_H^{\alpha,\sigma} \circ P,
$$
where the linking transformation is given by $Pf(x) = \sqrt{x}f(x^2)$, $x>0$.
We now see that the $L^p(U) \to L^q(V)$ estimate for $\mathcal{I}_S^{\alpha,\sigma}$ is equivalent
to the $L^p(\widehat{U}) \to L^q(\widehat{V})$ estimate for $\mathcal{I}_H^{\alpha,\sigma}$,
where
$$
\widehat{U}(x) = U(x^2) x^{1-p\slash 2}, \qquad \widehat{V}(x) = V(x^2) x^{1-q\slash 2}.
$$
This allows to conclude immediately the following result from Proposition \ref{prop:ell}.
\begin{propo} \label{prop:mL} 
Assume that $d=1$ and $\alpha \ge -1\slash 2$. 
Let $\sigma > 0$, $1<p \le q < \infty$ and 
$$
A < \frac{1}{p'}+\frac{\alpha}{2}, \qquad 
B < \frac{1}{q} + \frac{\alpha}{2}, \qquad A+B \ge \alpha \bigg(\frac{1}{p}-\frac{1}{q}\bigg).
$$
\begin{itemize}
\item[(i)]
If $\sigma \ge \alpha+1$, then $\mathcal{I}^{\alpha,\sigma}_S$ maps boundedly 
$L^p(\mathbb{R}_+,x^{A p})$ into $L^q(\mathbb{R}_+,x^{-B q})$. 
\item[(ii)]
If $\sigma < \alpha+1$, then the same boundedness is true under the additional condition
\begin{equation} \label{cnd72}
\frac{1}{q} \ge \frac{1}{p} + A+B - \sigma.
\end{equation}
\end{itemize}
\end{propo}
Taking into account the fact that the kernel $\mathcal{K}_S^{\alpha,\sigma}(x,y)$ is decreasing
with respect to the index $\alpha \ge 0$ (this follows 
from the analogous monotonicity of the involved Bessel function $I_{\alpha}$, see for instance
the proof of \cite[Proposition 2.1]{NS2}) we get the following.
\begin{corollary} \label{cor:mL} 
Assume that $d=1$ and $\alpha \ge 0$. 
Let $\sigma > 0$, $1<p \le q < \infty$ and 
$$
A < \frac{1}{p'}, \qquad B < \frac{1}{q}, \qquad A+B \ge 0.
$$
\begin{itemize}
\item[(i)]
If $\sigma \ge 1$, then $\mathcal{I}^{\alpha,\sigma}_S$ maps boundedly 
$L^p(\mathbb{R}_+,x^{A p})$ into $L^q(\mathbb{R}_+,x^{-B q})$. 
\item[(ii)]
If $\sigma < 1$, then the same boundedness is true under condition \eqref{cnd72}.
\end{itemize}
\end{corollary}
Now, with the aid of a suitable multiplier theorem (for example, the result of D\l{}ugosz \cite{Dl}
for integer $\alpha$ combined with Kanjin's transplantation theorem \cite{Ka} gives sufficiently general
multiplier theorem for all $\alpha>-1$),
the aforementioned result of Kanjin and Sato can be recovered as a special case of
Proposition \ref{prop:mL} for $\alpha \in [-1\slash 2, 0]$, or Corollary \ref{cor:mL} when $\alpha \ge 0$.
On the other hand, it is perhaps interesting to note that this result does not follow
from Theorem \ref{thm:lag_Hw}. 
We remark that the result of Kanjin and Sato for the full range of $\alpha \in (-1,\infty)$
can be recovered by specifying Proposition \ref{prop:mL} to $\alpha=0=A=B$, using a multiplier
theorem, and then applying Kanjin's transplantation theorem.

Another comment concerns the very recent paper by Bongioanni and Torrea \cite{BT2}.
The authors obtain there, among many other results, power weighted $L^p$-boundedness of potential
operators related to one-dimensional Laguerre function systems, see 
\cite[Theorem 7, Proposition 2]{BT2} and remarks closing \cite[Section 5]{BT2}.
Our present results contain those of \cite{BT2} on mapping properties of Laguerre potentials
as special cases, at least when $\alpha \ge -1\slash 2$.
Indeed, for such $\alpha$ \cite[Theorem 7]{BT2} follows by specifying
Proposition \ref{prop:mL} to $p=q$ and $B=-A$.

Still another comment explains, to some extent, why $L^p-L^q$ boundedness of the Hermite potential 
operator $\mathcal{I}^{\sigma}$ should be expected for all $p,q\in[1,\infty]$, provided that $\sigma$ 
is large. This 'phenomenon' should be linked with the ultracontractivity property of the Hermite 
semigroup $\{e^{-t\mathcal H}\}_{t>0}$: given $p,q\in[1,\infty]$, $e^{-t\mathcal H}$ maps $L^p(\R)$ 
into  $L^q(\R)$ boundedly. Thus we can expect a similar property for the average 
$\mathcal H^{-\sigma}=\Gamma(\sigma)^{-1}\int_0^\infty e^{-t\mathcal H}t^{\sigma-1}\,dt$, 
of course having a hope for a good control of the corresponding operator norms of $e^{-t\mathcal H}$. 
This indeed happens. For $t\ge1$ a crude estimate based on bounds of $L^p$-norms of the Hermite 
functions gives $\|e^{-t\mathcal H}\|_{L^p\to L^q}\le C_{pq}e^{-td}$, $t\ge1$,  $p,q\in[1,\infty]$ 
(see \cite[Remark 2.11]{ST1}), hence 
$\int_1^\infty \|e^{-t\mathcal H}\|_{L^p\to L^q}t^{\sigma-1}\,dt<\infty$ follows for $\sigma>0$. 
On the other hand, we will show that for  $p,q\in[1,\infty]$,
\begin{equation}\label{finn} 
\|e^{-t\mathcal H}\|_{L^p\to L^q}\le C_{pq}t^{-d/2},\qquad 0<t\le 1,
\end{equation}
hence  $\int_0^1 \|e^{-t\mathcal H}\|_{L^p\to L^q}t^{\sigma-1}\,dt<\infty$ for $\sigma>d/2$, 
and thus $\|\mathcal H^{-\sigma}\|_{L^p\to L^q}<\infty$ for $\sigma>d/2$. To verify \eqref{finn} 
it is sufficient to consider only the cases $p=\infty,q=1$ and $p=1,q=\infty$; in the cases $p=q=1$ 
and $p=q=\infty$, \eqref{finn} is trivially satisfied ($\{e^{-t\mathcal H}\}_{t>0}$ is a semigroup of
contractions on $L^p(\R)$, $1\le p\le\infty$, see \cite[Remark 2.10]{ST1}), and then interpolation does 
the job.  Checking \eqref{finn} is easy for $p=1,q=\infty$: when estimating the relevant operator norm 
it suffices to use the estimate $G_t(x,y)\le Ct^{-d/2}$. Proving \eqref{finn} for $p=\infty,q=1$, 
reduces to checking that
\begin{equation*}
\int_{\R}\int_{\R} \exp\Big(-\frac18t\|x+y\|^2-\frac1{4t}\|x-y\|^2\Big)dx\,dy\le C,\qquad 0<t\le 1.
\end{equation*}
Since the variables may be separated it is sufficient to consider the case $d=1$. 
Then the desired estimate follows by the change of variables $u=x+y$, $v=x-y$.

Finally, we comment on relations between the classical Bessel potentials in $\R$ and the potential
operators studied in this paper, and point out important consequences of these relations.
The Bessel potentials $J^{\sigma}$ are associated to the powers $(\textrm{Id}-\Delta)^{-\sigma\slash 2}$
in the analogous way as the Riesz potentials $I^{\sigma}$ correspond to $(-\Delta)^{-\sigma\slash 2}$,
see \cite{AS}.
For $\sigma > 0$ they are given as convolutions
$$
J^{\sigma} = \mathcal{G}^{\sigma} * f
$$
with the radial kernel
$$
\mathcal{G}^{\sigma}(x) = K_{\frac{d-\sigma}{2}}\big(\|x\|\big) \|x\|^{\frac{\sigma-d}{2}}, \qquad x \in \R,
$$
where $K_{\nu}$ is the modified Bessel function of the third kind and order $\nu$, usually referred to as
McDonald's function, cf. \cite[Chapter 5]{Leb}.
The kernels $\mathcal{G}^{\sigma}$
share many properties of the Riesz kernels
$\|x\|^{\sigma-d}$ (this in particular concerns positiveness), but, in contrast with the Riesz kernels,
they decay exponentially at infinity; this last feature makes a significant difference in mapping properties
of the two kinds of potentials. By basic properties of McDonald's function 
(see \cite[Ch.\,2, Sec.\,3]{AS}) it is easily seen that $\mathcal{G}^{2\sigma}$ behaves very similarly
to the majorization kernel $K^{\sigma}$ used in the Hermite setting in Section \ref{sec:herm}.
In fact, given $\sigma >0$, $K^{\sigma}$ can be controlled pointwise by $\mathcal{G}^{2\sigma}$,
see \cite[Ch.\,2, Sec.\,4]{AS}. Thus, for nonnegative functions $f$ in $\R$,
$$
\mathcal{I}^{\sigma}f(x) \lesssim J^{2\sigma}f(x) \lesssim I^{2\sigma}f(x), \qquad x \in \R,
$$
and therefore any weighted $L^p-L^q$ boundedness of $J^{2\sigma}$ (or $I^{2\sigma}$)
is inherited by the Hermite
potential operator $\mathcal{I}^{\sigma}$. Clearly, by the methods presented in this paper,
further consequences of this concern (a bit less explicitly) the Laguerre potentials
$\mathcal{I}_H^{\alpha,\sigma}$, $\mathcal{I}^{\alpha,\sigma}$ and $\mathcal{I}_S^{\alpha,\sigma}$,
and the Dunkl potentials.


\end{document}